\documentclass[11pt]{article}
\usepackage[margin=1in]{geometry}
\usepackage{amsmath,amssymb,amsthm,mathtools,enumitem,bm}
\usepackage{mathrsfs}
\usepackage[T1]{fontenc}
\usepackage{lmodern}
\usepackage{microtype}
\usepackage{hyperref}
\usepackage{comment}

\title{Abundance of Bergman metrics with constant positive holomorphic sectional curvature} %Positive integral multiples of the Fubini–Study metric are locally Bergman metrics in high dimensions}
\author{}
\date{}

\theoremstyle{plain}
\newtheorem{theorem}{Theorem}[section]
\newtheorem*{theorem*}{Theorem}
\newtheorem{corollary}{Corollary}[section]
\newtheorem{lemma}{Lemma}[section]
\newtheorem{proposition}{Proposition}[section]

\theoremstyle{definition}
\newtheorem{definition}{Definition}[section]
\newtheorem*{definition*}{Definition}

\theoremstyle{remark}
\newtheorem{remark}{Remark}[section]

\numberwithin{equation}{section}

\newenvironment{customtheorem}[1]
{\par\noindent\textbf{Theorem #1.} \itshape}
{\par}

\newcommand{\C}{\mathbb C}
\newcommand{\R}{\mathbb R}
\newcommand{\N}{\mathbb N}
\newcommand{\Z}{\mathbb Z}
\newcommand{\pj}{\mathbb P}

\newcommand{\om}{\Omega}
\newcommand{\dd}{\,d}
\newcommand{\Ol}{\overline}
\newcommand{\al}{\alpha}
\newcommand{\inner}[2]{\langle #1, #2\rangle}
\newcommand{\I}{\mathcal I_m}
\newcommand{\A}{A^2_{(n,0)}}
\newcommand{\leqB}{\preceq_B}
\newcommand{\simB}{\sim_B}
\newcommand{\Res}{\mathcal{R}}

\newcommand{\U}{\mathrm{U}}

\begin{document}
\author{
Shreedhar Bhat, Soumya Ganguly, Achinta Kumar Nandi, Ming Xiao\thanks{M.~Xiao is supported in part by NSF grant DMS-2045104.}
}
\maketitle

\begin{abstract}
An outstanding open question, which has attracted renewed attention following the pioneering work of Huang--Li--Treuer, is whether, for a given positive integer $m$, there exists a complex manifold whose Bergman metric is locally isometric to $m$ times the Fubini--Study metric. Previously, this question had only been resolved in the case $m=1$.
In this paper, we construct, for any pair of positive integers $(m,n)$ with $n \geq 2$, an $\mathbb{R}$-parameter (hence uncountable) family of Reinhardt domains in $\mathbb{C}^n$ whose Bergman metrics are all locally isometric to $m$ times the Fubini--Study metric. Moreover, we show that the domains in this family are mutually Bergman inequivalent. 
This not only answers the folklore question, but also suggests that a reasonable classification of the geometry of such complex manifolds is infeasible. We also note such examples cannot exist in dimension one. 
The results complete the remaining open case in the study of complex manifolds whose Bergman space separates points and whose Bergman metric has constant holomorphic sectional curvature. Our approach differs from existing methods in the literature. We reduce the construction to a mapping problem and apply a Brouwer fixed point argument to establish the existence of the desired domains.
\end{abstract}

\medskip


\section{Introduction}\label{sec1}

Since the classical works of Bergman \cite{Be33, Be35} and Kobayashi \cite{Ko59}, the Bergman kernel and metric have played a fundamental role in complex analysis and complex geometry. Bergman's original notion~\cite{Be33} was given for domains in $\mathbb{C}^n$, and was later extended by Kobayashi \cite{Ko59} to general complex manifolds. We recall Kobayashi's notion of the Bergman kernel on a complex manifold from \cite{Ko59}. Let $M$ be an $n$-dimensional complex manifold.  (Throughout the paper, complex manifolds are assumed to be connected.) Let $A^2_{(n,0)}(M)$ denote the space of holomorphic $(n,0)$-forms $f$ such that $|\int_M f \wedge \overline{f}| < \infty$. This is a Hilbert space with inner product
$$\langle f, g \rangle = (\sqrt{-1})^{n^2} \int_M f \wedge \overline{g}, \quad f, g \in A^2_{(n,0)}(M).$$
Given an orthonormal basis $\{\phi_k\}_{k=0}^N$ of $A^2_{(n,0)}(M)$, where $0 \leq N \leq \infty$, the Bergman kernel of $M$ is the $(n,n)$-form on $M \times M$ defined by
$K_M=\sum_k \phi_k \wedge \overline{\phi_k},$
which is independent of the choice of orthonormal basis. In local holomorphic coordinates $(z_1,\cdots,z_n)$ and $(w_1,\cdots,w_n)$, the kernel can be written as
\begin{equation}\label{eqnbergmanlocal}
	K_M = \widetilde{K}_M (z,w) \, dz_1 \wedge \cdots \wedge dz_n \wedge d\bar{w}_1 \wedge \cdots \wedge d\bar{w}_n,
\end{equation}
where $\widetilde{K}_M(\cdot,\cdot)$ is sesqui-holomorphic, i.e., holomorphic in the first variable and anti-holomorphic in the second. We also often restrict $K_M$ to the diagonal $z=w$ and refer to this restriction as the Bergman kernel of $M$.
%If $\widetilde{K}_M$ is everywhere positive and the associated K\"ahler form $\omega=i\partial\bar{\partial} \log \widetilde{K}_M$
%is positive-definite (these properties are independent of the choice of local coordinates), then the induced K\"ahler metric is called the Bergman metric on $M$.
When \( M \) is a domain in $\mathbb{C}^n$, the Bergman kernel \( K_M \) can be identified, in the standard Euclidean coordinates, with its scalar-valued representing function \( \widetilde{K}_M \); moreover, \( A^2_{(n,0)}(M) \) is naturally identified with the space \( A^2(M) \) of square-integrable holomorphic functions.  Kobayashi \cite{Ko59} further introduced the following natural conditions on the Bergman space of a complex manifold:

\begin{definition}\label{base point free and separate holomorphic directions}
	(Kobayashi \cite{Ko59})
	Let $M$ be an $n$-dimensional complex manifold. The Bergman space
	$A^2_{(n,0)}(M)$ is said to
	\begin{enumerate}
		\item[(1)] be \emph{base point free} if, for every $\xi \in M$, there exists
		$f \in A^2_{(n,0)}(M)$ such that $f(\xi)\neq 0$;
		
		\item[(2)] \emph{separate holomorphic directions} if, for every
		$\xi \in M$ and every nonzero vector $Z \in T^{1,0}_{\xi}M$, there exists
		$f \in A^2_{(n,0)}(M)$ such that $f(\xi)=0$ and $Zf(\xi)\neq 0$;
		
		\item[(3)] \emph{separate points} if, for any two distinct points
		$\xi_1,\xi_2 \in M$, there exists $f \in A^2_{(n,0)}(M)$ such that
		$f(\xi_1)=0$ and $f(\xi_2)\neq 0$.
	\end{enumerate}
\end{definition}

As shown by Kobayashi \cite{Ko59}, conditions~(1) and~(2) are necessary and sufficient for $M$ to admit a Bergman metric, namely the metric induced by the K\"ahler form given in local coordinates by
$\omega(z)=i\partial\bar{\partial} \log \widetilde{K}_M(z, z).$
This expression is independent of the choice of local coordinates. Condition (3) holds for many classes of interest, in particular for bounded domains in $\mathbb{C}^n$. Since this paper focuses on complex manifolds whose Bergman space satisfies all three conditions in Definition~\ref{base point free and separate holomorphic directions}, we introduce the following terminology for convenience.

\begin{definition}
	A complex manifold $M$ is called \emph{Bergman nondegenerate} if its Bergman space $A^2_{(n,0)}(M)$ is base point free, separates holomorphic directions, and separates points; that is, it satisfies the three conditions in Definition~\ref{base point free and separate holomorphic directions}.
\end{definition}

In particular, all bounded domains in $\mathbb{C}^n$ are Bergman nondegenerate.
One of the most basic models for the Bergman kernel and metric is the unit ball $\mathbb{B}^n$ in $\mathbb{C}^n$.
It is a standard fact that the Bergman metric of the unit ball $\mathbb{B}^n$ has constant holomorphic sectional curvature $-2(n+1)^{-1}$.
In 1965, Lu \cite{Lu65} proved his well-known  characterization theorem of the unit ball in terms of  the Bergman metric.

\medskip

%\begin{theorem*}[Lu's uniformization theorem, \cite{L65}]
%\begin{theorem}{0}{\rm (Lu's uniformization theorem, \cite{L65}).}
	{\it Let $D$ be a bounded domain in $\mathbb{C}^n$ such that its Bergman  metric $g$ is complete.  The Bergman  metric $g$ has constant holomorphic sectional curvature $\tau$ if and only if $D$ is biholomorphic to the unit ball $\mathbb{B}^n$. In this case, the constant $\tau$ equals $-2(n+1)^{-1}$.}
%\end{theorem}

\medskip

Lu's theorem has important applications to the study of the geometry of complex domains via the Bergman metric; see the discussions in \cite{HLT, ETX1, ETX2}. It is also connected to many active areas of current research. A closely related topic is the study of Lu Qi-Keng domains, concerning the zero-free property of the Bergman kernel; see, for instance, \cite{Bo86, BFS99, Ha98} and the references therein, as well as the survey article \cite{Bo00} for many  developments along this line. Another independent but related direction is the study of rigidity and global extension problems for local holomorphic isometric maps with respect to Bergman metrics. In this direction, beginning with the seminal work of Mok \cite{Mo1, Mo2} (see also \cite{Mo09, MN09}), a substantial body of research has been devoted to local holomorphic isometric embeddings between domains, possibly of different dimensions, leading to many far-reaching developments. As these directions lie beyond the scope of the present paper, we do not attempt to survey the extensive literature here.

\medskip

In this paper, we focus on Bergman metrics with constant holomorphic sectional curvature. Lu's theorem admits a natural extension to Bergman nondegenerate complex manifolds with complete Bergman metrics; see the survey discussion in \cite[Section~2]{ETX2} and the references therein. On the other hand, a longstanding folklore problem asks for an understanding of Bergman metrics with constant holomorphic sectional curvature without assuming completeness. More precisely, the following question remains open.

\medskip

\noindent
\noindent
{\bf Question 1.}
Let $M$ be an $n$-dimensional Bergman nondegenerate complex manifold whose (not necessarily complete) Bergman metric has constant holomorphic sectional curvature $\tau$. Determine the set $\mathcal{H}_n$ of all possible values of $\tau$. Furthermore, for a given $\tau \in \mathcal{H}_n$, is it possible to classify all such manifolds $M$ realizing this curvature?

\medskip

Bergman showed in \cite{Be48} that the holomorphic sectional curvature of the Bergman metric of a bounded domain in $\mathbb{C}^n$ is bounded above by that of the Fubini--Study metric, namely by $2$. This upper bound was later extended to general complex manifolds by Kobayashi~\cite{Ko59}. Consequently, $\mathcal{H}_n \subseteq (-\infty,2]$.
Apart from this, in the absence of a completeness assumption, it is already subtle to determine even the possible signs of $\tau$. In the following, we begin by reviewing the known results and then resolve the remaining open case of Question~1.
\medskip

The first major advances were obtained by Dong--Wong \cite{DW22, DW-AJM} and Huang--Li--Treuer \cite{HLT}. Further research was subsequently carried out by Ebenfelt, Treuer, and the last author \cite{ETX1}, as well as in the works \cite{ETX2, DWW25, HL-preprint}. We next survey the results from these works that are most relevant to the present paper. The discussion is organized according to the sign of $\tau$, as the phenomena differ substantially in the cases $\tau=0$, $\tau<0$ and $\tau>0$.

We begin with the case $\tau=0$. Huang--Li--Treuer \cite{HLT} proved that there exists no Bergman nondegenerate complex manifold whose Bergman metric has identically zero holomorphic sectional curvature. Consequently,
$0 \notin \mathcal{H}_n.$

\medskip

We next consider the case $\tau<0$. By the work of Huang--Li \cite{HLT, HL-preprint}, if $M$ is a Stein manifold whose Bergman metric has constant negative holomorphic sectional curvature $\tau$, then $M$ is biholomorphic to the unit ball $\mathbb{B}^n$ with a relatively closed pluripolar set removed. Ebenfelt, Treuer, and the last author \cite{ETX1} subsequently studied this problem for Bergman nondegenerate complex manifolds without assuming Steinness. To describe their result, we recall the notion of Bergman negligible subsets, which was introduced in \cite{ETX1} for bounded domains in $\mathbb{C}^n$. For the purposes of the present paper, we extend this notion to general complex manifolds.

%\begin{definition}
%Let $M$ be a complex manifold with nontrivial Bergman space, and let $E \subset M$ be a (possibly empty) closed proper subset.  We say that $E$ is a \emph{Bergman negligible subset} of $M$ if every $\omega \in A^2_{(n,0)}(M \setminus E)$ extends holomorphically to some $\widetilde{\omega} \in A^2_{(n,0)}(M)$, and the extension preserves the $L^2$ norm. This is equivalent to the statement that the restriction map
	%\[
	%\mathrm{Res} : A^2_{(n,0)}(M) \longrightarrow A^2_{(n,0)}(M \setminus E), \qquad \omega \longmapsto \omega|_{M \setminus E},
	%\]
	%is a unitary isomorphism.
%\end{definition}

\begin{definition}\label{defnbergmannegligible}
Let $M$ be a complex manifold with nontrivial Bergman space, i.e., $A^2_{(n,0)}(M) \neq \{0\}$, and let $E \subseteq M$ be a (possibly empty) closed proper subset. 
We say that $E$ is a \emph{Bergman negligible subset} of $M$ if every $\psi \in A^2_{(n,0)}(M \setminus E)$ extends to a form $\phi \in A^2_{(n,0)}(M)$ with preservation of the $L^2$ norm. Equivalently, the restriction map
\[
\Res : A^2_{(n,0)}(M) \longrightarrow A^2_{(n,0)}(M \setminus E), \qquad \phi \longmapsto \phi|_{M \setminus E},
\]
is a unitary isomorphism.
\end{definition}

\medskip

We remark that any closed pluripolar subset is Bergman negligible (cf.\ \cite[Lemma 1]{Ir}). However, the converse does not hold in general; see \cite{ETX1, ETX2} for further discussion. We also record several basic properties of Bergman negligible subsets in the following remarks, whose proofs will be given in \S\ref{sec41}.

\begin{remark}\label{rmkbergmanneg1}
Let $M$ be a complex manifold with nontrivial Bergman space. If $E$ is a Bergman negligible subset of $M$, then $M \setminus E$ is open, dense, and connected in $M$. In particular, $E$ has empty interior.
\end{remark}

\begin{remark}\label{rmkbergmanneg2}
	We have the following characterization of Bergman negligible subsets. Let $M$ be a complex manifold with nontrivial Bergman space, and let $M'$ be a domain in $M$. Then $M \setminus M'$ is Bergman negligible in $M$ if and only if the Bergman kernel of $M'$ is the restriction of that of $M$ to $M'$. This was proved in \cite{ETX1} for bounded domains in $\C^n$; we will establish the general case in \S\ref{sec41}.
\end{remark}

To prepare for our subsequent study, we introduce a natural relation $\leqB$ associated with Bergman negligible subsets, as well as the equivalence relation it generates.

\begin{definition}\label{defnleqb}
Let $M_1$ and $M_2$ be complex manifolds with nontrivial Bergman spaces. We write $M_1 \leqB M_2$ if there exists a Bergman negligible subset $E \subseteq M_2$ such that $M_1$ is biholomorphic to $M_2 \setminus E$.
\end{definition}

The equivalence relation generated by $\leqB$ (i.e., the smallest equivalence relation that contains $\leqB$) will be called \emph{Bergman equivalence}. More precisely, we have the following definition.

\begin{definition}
Let $M_1$ and $M_2$ be complex manifolds with nontrivial Bergman spaces. We say that $M_1$ and $M_2$ are \emph{Bergman equivalent}, and write $M_1 \simB M_2$, if there exists a finite sequence of complex manifolds
	\[
	M_1 = X_1,\; X_2,\; \cdots,\; X_k = M_2
	\]
such that for each $j=1,\cdots,k-1$, either $X_j \leqB X_{j+1}$ or $X_{j+1} \leqB X_j$. If no such sequence exists, we say that $M_1$ and $M_2$ are \emph{Bergman inequivalent}. For a complex manifold $M$ with nontrivial Bergman space, we denote its Bergman equivalence class by $[M]_B$.
\end{definition}

Roughly speaking, $M_1$ and $M_2$ are Bergman equivalent if one can be obtained from the other by finitely many operations consisting of biholomorphic transformations, deletion of Bergman negligible subsets, and their reinsertion. These operations do not change the Bergman space. We record the following basic properties of Bergman equivalent complex manifolds that will be used later.

\begin{remark}\label{rmkbergmanequivalence}
	Let $M_1$ and $M_2$ be Bergman equivalent complex manifolds with nontrivial Bergman spaces. Then there exists a unitary isomorphism between their Bergman spaces. Moreover, there exist open dense subsets $D_1 \subseteq M_1$ and $D_2 \subseteq M_2$, and a biholomorphism $f: D_1 \to D_2$ such that the Bergman kernel forms of $M_1$ and $M_2$ are preserved under $f$. A more precise formulation of this statement, together with its proof, is given in Proposition~\ref{prop:rmk18} of \S\ref{sec41}.
\end{remark}

We now return to the results of Ebenfelt--Treuer--Xiao \cite{ETX1}. By Theorem 3.1 in \cite{ETX1}, the Bergman metric of an $n$-dimensional Bergman nondegenerate complex manifold $M$ has constant negative holomorphic sectional curvature $\tau$ if and only if $M \leqB \mathbb{B}^n$. In particular, $\tau = -\frac{2}{n+1}$.

For convenience, we fix the following notation, which will be used throughout the paper, and then summarize the known results discussed above.

\medskip

%\noindent
%{\bf Notation.}
\begin{itemize}
	\item $\mathcal{H}_n$ denotes the set of $\tau \in \mathbb{R}$ for which there exists an $n$-dimensional Bergman nondegenerate complex manifold with Bergman metric of constant holomorphic sectional curvature $\tau$.
	
	\item For $\tau \in \mathcal{H}_n$, $\mathcal{M}_n(\tau)$ denotes the set of $n$-dimensional Bergman nondegenerate complex manifolds with Bergman metric of constant holomorphic sectional curvature $\tau$.
	
	\item We set $\mathcal{E}_n(\tau)=\{[M]_B : M \in \mathcal{M}_n(\tau)\}$.
\end{itemize}

\medskip

By the results in \cite{HLT} and \cite{ETX1} discussed above,  the following holds.

\medskip

\begin{customtheorem}{0}(\cite{HLT, ETX1})
	We have $0 \notin \mathcal{H}_n$, $\mathcal{H}_n \cap (-\infty,0)=\left\{-\frac{2}{n+1}\right\}$, and $\mathcal{E}_n\!\left(-\frac{2}{n+1}\right)=\{[\mathbb{B}^n]_B\}$.
\end{customtheorem}
\medskip

%for $\tau <0,$ we have $\mathcal{E}_n(-\frac{2}{n+1})=\{[\mathbb{B}^n]_B\},$ and $\mathcal{E}_n(\tau)=\emptyset$ for all other negative $\tau.$

We finally consider the case $\tau>0$. Note that the representing function of the Bergman kernel in~\eqref{eqnbergmanlocal}, when restricted to the diagonal, is locally a sum of squared moduli of holomorphic functions. Using this, one can show that a Bergman metric cannot be locally isometric to a non-integer multiple of the Fubini--Study metric (see \cite{Lu65}, or the proof of Theorem 3.2 in \cite{HLT}). It follows that if a Bergman metric has positive constant holomorphic sectional curvature $\tau$, then $\tau$ must coincide with that of some integer multiple of the Fubini--Study metric. That is, $\tau=\frac{2}{m}$ for some positive integer $m$. Consequently, $\mathcal{H}_n \cap (0,\infty)$ is contained in the discrete set $\left\{\frac{2}{m}: m \in \mathbb{Z}_+\right\}$. Here $\mathbb{Z}_+$ denotes the set of positive integers.

Huang--Li \cite{HLT} proved that if $M$ is a Bergman nondegenerate complex manifold whose Bergman metric has positive constant holomorphic sectional curvature, then the Bergman space of $M$ is finite-dimensional and $M$ is biholomorphic to a domain in complex projective space $\mathbb{P}^n$. Furthermore, in \cite[Section 3]{HLT}, Huang--Li constructed an unbounded Reinhardt domain in $\mathbb{C}^2$ whose Bergman metric has constant holomorphic sectional curvature $\tau=2$. This provides the first example of a Bergman metric with positive constant holomorphic sectional curvature.
In very recent work \cite{HK26}, Hayashida and Kamimoto extended this construction to higher dimensions, showing that there exist Reinhardt domains in $\mathbb{C}^n$, $n \geq 2$, whose Bergman metrics have constant holomorphic sectional curvature $2$. Consequently, by \cite{HLT, HK26}, $2 \in \mathcal{H}_n$ for $n \geq 2$.

On the other hand, it remains open whether all fractions $\frac{2}{m}$, $m \in \mathbb{Z}_+$, arise in $\mathcal{H}_n$. As will be explained at the end of this section, the case $m>1$ leads to a fundamentally different and substantially more involved problem than the case $m=1$.  Before addressing this question, we begin with the following observations:

\begin{remark}\label{rmkngreaterthan1}
There exists no domain in $\C$ whose Bergman metric has constant positive holomorphic sectional curvature. Moreover, there exists no Bergman nondegenerate Riemann surface whose Bergman metric has constant positive holomorphic sectional curvature. As a consequence, $\mathcal{H}_1 \cap (0,\infty)=\emptyset$. We will prove these statements in \S\ref{sec41}.
\end{remark}

Therefore, by Remark~\ref{rmkngreaterthan1}, it suffices to consider Bergman nondegenerate complex manifolds whose Bergman metric has constant positive holomorphic sectional curvature in dimension $n \geq 2$. Our first main result resolves the above question by showing that all fractions $\frac{2}{m}$, $m \in \mathbb{Z}_+$, arise in $\mathcal{H}_n$ for $n \geq 2$.
Throughout the paper, $\mathbb{N}$ denotes the set of nonnegative integers. For a multi-index $\alpha=(\alpha_1,\cdots,\alpha_n) \in \mathbb{N}^n$, we write $z^\alpha = z_1^{\alpha_1} \cdots z_n^{\alpha_n}$ and $|\alpha|=\alpha_1+\cdots+\alpha_n$.

%Roughly speaking, $M_1$ and $M_2$ are Bergman equivalent if one can be recovered from the other by operations of a biholomorphic tranformation, deletion of a Bergman negligible set, and filling one back in, finitely many times.

\begin{theorem}\label{thm1}
	For any pair of positive integers $(m,n)$ with $n \geq 2$, there exists a constant $C_{m,n}$, depending only on $m$ and $n$, such that for every $C > C_{m,n}$, there exists an unbounded Reinhardt domain $\Omega_C \subseteq \mathbb{C}^n$ with the following properties:
	
	\medskip
	
	(1) The Bergman space of $\Omega_C$ is spanned by $\{z^{\alpha} : \alpha \in \mathbb{N}^n,\ |\alpha| \leq m\}$. In particular, the Bergman space separates points.
	
	\medskip
	
	(2) The Bergman kernel of $\Omega_C$ is given by $K_C(z,z)=C(1+\|z\|^2)^m$. Consequently, the Bergman metric of $\Omega_C$ coincides with the restriction to $\Omega_C \subseteq \mathbb{C}^n \subseteq \mathbb{P}^n$ of $m$ times the Fubini--Study metric.
\end{theorem}

Indeed, property (2) implies property (1) by a straightforward application of D'Angelo's lemma \cite{DA93} or Calabi's rigidity theorem \cite{Ca53}. However, we include both properties in Theorem~\ref{thm1} for consistency with the formulation of Question~1. A natural next question is whether it is possible to classify the domains satisfying both properties (1) and (2), up to Bergman equivalence. To this end, we fix the following notation.
For any pair of positive integers $(m,n)$ and any positive constant $C$, let $\mathcal{F}(m,n;C)$ denote the set of domains $\Omega \subseteq \mathbb{C}^n$ whose Bergman kernel is given by $K_\Omega(z,z)=C(1+\|z\|^2)^m$ (and hence whose Bergman space separates points). Then, by Remark~\ref{rmkngreaterthan1}, one has $\mathcal{F}(m,1;C)=\emptyset$ for every positive integer $m$. When $n \geq 2$, Theorem~\ref{thm1} shows that there exists a constant $C_{m,n}>0$ such that $\mathcal{F}(m,n;C) \neq \emptyset$ for all $C > C_{m,n}$. We further establish the following theorem.

\begin{theorem}\label{thm2}
	For any pair of positive integers $(m,n)$ with $n \geq 2$, the following statements hold.
	
	\begin{enumerate}
		\item[{\rm (1)}] Let $C$ and $C'$ be positive numbers, and let $\Omega \in \mathcal{F}(m,n;C)$ and $\Omega' \in \mathcal{F}(m,n;C')$. If $C\neq C'$, then $\Omega$ and $\Omega'$ are Bergman inequivalent.
		
		\item[{\rm (2)}] %Let $C_{m,n}$ be the constant in Theorem~\ref{thm1}. Then for every $C > C_{m,n}$,  $\mathcal{F}(m,n;C)$ contains a one-parameter family of unbounded Reinhardt domains $\{\Omega_C(t) : t \in (0,1)\}$ with the following property:  For any (not necessarily distinct) $C, C'> C_{m,n}$ and  $t,t' \in (0, 1), \Omega_C(t) \cap \Omega_{C'}(t') \neq \emptyset$. 
		%Besides, the domains $\Omega_C(t)$ and $\Omega_C(t')$ are Bergman inequivalent if $t \neq t'$.
		Let $C_{m,n}$ be the constant in Theorem~\ref{thm1}. Then for every $C > C_{m,n}$,  $\mathcal{F}(m,n;C)$ contains a one-parameter family of \emph{mutually Bergman inequivalent} (unbounded Reinhardt) domains $\{\Omega_C(\mu) : \mu \in (0,1)\}$ with the following property:  For any (not necessarily distinct) $C, C'> C_{m,n}$ and  $\mu,\mu' \in (0, 1),~~ \Omega_C(\mu) \cap \Omega_{C'}(\mu') \neq \emptyset$. 
		%Besides, the domains $\Omega_C(t)$ and $\Omega_C(t')$ are Bergman inequivalent if $t \neq t'$.

	\end{enumerate}
\end{theorem}

The preceding two theorems show that there exists an uncountable family of mutually Bergman inequivalent unbounded domains in $\mathbb{C}^n$ whose Bergman kernels are given by the same polynomial, or by the same polynomial up to a multiplicative constant. (Moreover, these domains have nontrivial pairwise intersections.) This reveals a striking contrast between the behavior of Bergman kernels on unbounded and bounded domains: Such a phenomenon cannot occur for bounded domains, as shown by the following theorem and remark. For two sets $E_1$ and $E_2$, we write $E_1 \Delta E_2$ for their symmetric difference. That is, $E_1 \Delta E_2$ equals the union of $E_1 \setminus E_2$ and $E_2 \setminus E_1$.

\begin{theorem}\label{thm20}
	Let $D_1, D_2$ be bounded domains in $\mathbb{C}^n$ with Bergman kernels $K_1, K_2$, respectively. Assume that there exist a domain $\hat{D} \subseteq \mathbb{C}^n$ containing $D := D_1 \cup D_2$, a positive real-analytic function $K$ on $\hat{D}$, and a constant $\lambda > 0$ such that:
	\begin{itemize}
		\item[(a)] $\omega:=i\partial\bar{\partial} \log K$ defines a positive-definite K\"ahler form on $\hat{D}$;
		\item[(b)] $K_1 = K$ on $D_1$, and $K_2 = \lambda K$ on $D_2$.
	\end{itemize}
Then $\lambda = 1$; and $D_1 \Delta D_2$ has Lebesgue measure zero. If in addition $K$ admits a sesqui-holomorphic extension $\widetilde K(\cdot,\cdot)$ to $\hat{D} \times  \hat{D},$ then $D_1 \leqB D$ and $D_2 \leqB D$, and thus $D_1$ and $D_2$ are Bergman equivalent. 
%Moreover, , so in particular $D$ is connected, and one has $D_1 \leqB D$ and $D_2 \leqB D$.
\end{theorem}

%Here a function $\widetilde K(\cdot, \cdot)$ is called  sesqui-holomorphic if it is holomorphic in the first variable and anti-holomorphic in the second variable. Let $\{\Omega_C(\mu) : \mu \in (0,1)\}$ be the family of domains in Theorem~\ref{thm2}. 

\begin{remark}\label{rmkintersectsamek}
We have the following immediate consequence of Theorem~\ref{thm20}. Let $D_1, D_2$ be bounded domains in $\mathbb{C}^n$ with Bergman kernels $K_1, K_2$, respectively. Assume that $D_1 \cap D_2 \neq \emptyset$ and that $K_2 = \lambda K_1$ on $D_1 \cap D_2$ for some constant $\lambda > 0$. Then we have $\lambda=1$ and  $D_1 \Delta D_2$ has Lebesgue measure zero.
	
Indeed, in this case we take $\hat{D} = D := D_1 \cup D_2$, and define a function $K$ on $\hat{D}$ by setting $K = K_1$ on $D_1$ and $K = \frac{K_2}{\lambda}$ on $D_2$. Then $K$ is well defined on $\hat{D}$, and the first part of Theorem~\ref{thm20} applies.
\end{remark}

\medskip

Let $\{\Omega_C(\mu) : \mu \in (0,1)\}$, with $C > C_{m,n}$, be as in Theorem~\ref{thm2}. In contrast to Theorem~\ref{thm20} and Remark~\ref{rmkintersectsamek}, we will show that $\Omega_C(\mu) \Delta \Omega_C(\mu')$ has nonempty interior (and hence positive Lebesgue measure) whenever $\mu \neq \mu'$; see Remark~\ref{rmkdeltainteriorpt}.

Finally, combining our results with those of Huang--Li--Treuer \cite{HLT} and Ebenfelt--Treuer--Xiao \cite{ETX1} for the nonpositive curvature case as stated in Theorem~0, we settle Question~1.
%obtain an essentially complete answer to Question~1.

\begin{theorem}\label{thm3}
The following statements hold.
	
	\begin{enumerate}
		\item[{\rm (1)}] We have $\mathcal{H}_1=\{-1\}$. For every integer $n \geq 2$,
		\[
		\mathcal{H}_n=\left\{-\frac{2}{n+1}\right\} \cup \left\{\frac{2}{m}: m \in \mathbb{Z}_+\right\}.
		\]
		
		\item[{\rm (2)}] For every integer $n \geq 1$, one has $\mathcal{E}_n\!\left(-\frac{2}{n+1}\right)=\{[\mathbb{B}^n]_B\}$.
		
		\item[{\rm (3)}] For every pair of positive integers $(m,n)$ with $n \geq 2$, the set $\mathcal{E}_n\!\left(\frac{2}{m}\right)$ contains a one-parameter family of equivalence classes $\{[\Omega(t)]_B : t \in \mathbb{R}\}$, where $[\Omega(t)]_B \neq [\Omega(t')]_B$ whenever $t \neq t'$.
	\end{enumerate}
\end{theorem}

%In particular, part~(3) suggests that it is infeasible to classify the geometry of Bergman nondegenerate complex manifolds whose Bergman metric has constant positive holomorphic sectional curvature.
In particular, part~(3) shows that it is infeasible to classify the geometry of Bergman nondegenerate complex manifolds whose Bergman metric has constant positive holomorphic sectional curvature.
In the next remark, we reformulate the theorem for $n \geq 2$ to highlight the sharp contrast among the different curvature cases.

\begin{remark}
	Let $n \geq 2$. Then there is a sharp trichotomy in the zero, negative, and  positive curvature cases:
	\begin{itemize}
		\item (Nonexistence) $0 \notin \mathcal{H}_n$;
		
		\item (Rigidity) 
		$\mathcal{H}_n \cap (-\infty,0) = \left\{-\frac{2}{n+1}\right\}$, and 
		$\mathcal{E}_n\!\left(-\frac{2}{n+1}\right) = \{[\mathbb{B}^n]_B\}$;
		
		\item (Abundance)
		$\mathcal{H}_n \cap (0,\infty) = \left\{\frac{2}{m} : m \in \mathbb{Z}_+\right\}$, 
		and each $\mathcal{E}_n\!\left(\frac{2}{m}\right)$ contains a one-parameter family.
	\end{itemize}
\end{remark}

We next briefly discuss the difficulties and our strategy for proving Theorem~\ref{thm1}. To obtain the desired domain in Theorem~\ref{thm1} for a given pair of integers $(m,n)$, we aim to construct a Reinhardt domain $\Omega \subseteq \mathbb{C}^n$ satisfying the following two conditions. The first concerns the prescribed basis of the Bergman space:

%\noindent
%{\it (I) .~~The Bergman space of $\Omega$ is spanned by $\{z^{\alpha}: \alpha \in \N^n, |\alpha| \leq m\}$.}

\begin{enumerate}
	\renewcommand{\labelenumi}{(\Roman{enumi})}
	\renewcommand{\theenumi}{\Roman{enumi}}
	
	\item\label{condI}
	The Bergman space of $\Omega$ is spanned by $\{z^{\alpha}: \alpha \in \N^n, |\alpha| \leq m\}$.

\item[]{\hspace*{-\leftmargin}The second is a prescribed moment condition:}
	
	\item\label{condII}
	$\displaystyle \int_{\Omega} |z^\alpha|^2 \, \dd V_{2n}(z)
	= c \,\frac{\alpha!\,(m-|\alpha|)!}{m!}
	\quad \text{for all } \alpha \in \N^n \text{ with } |\alpha| \leq m$.
	
\end{enumerate}

%\medskip

%\noindent
%The second is a prescribed moment condition:

%\medskip

%\noindent
%{\it (II).~~$\displaystyle \int_{\Omega} |z^\alpha|^2 \, \dd V_{2n}(z)
	%= c \,\frac{\alpha!\,(m-|\alpha|)!}{m!}
	%\quad \text{for all } \alpha \in \N^n~\text{with}~|\alpha| \leq m$.}
	
%\begin{itemize}
	%\item[(II)]%\label{condII}
	%$\displaystyle \int_{\Omega} |z^\alpha|^2 \, \dd V_{2n}(z)
	%= c \,\frac{\alpha!\,(m-|\alpha|)!}{m!}
	%\quad \text{for all } \alpha \in \N^n~\text{with}~|\alpha| \leq m$.
%\end{itemize}

%\medskip

\noindent
Here $c$ is a constant independent of $\alpha$, and $V_{2n}$ denotes the Lebesgue measure on $\mathbb{C}^n$. Condition~(\ref{condI}) prescribes the Bergman space, and condition~(\ref{condII}) further prescribes the Bergman kernel required in Theorem~\ref{thm1}. The main difficulty lies in fulfilling condition~(\ref{condII}), and indeed most of the proof will be devoted to addressing this condition.

In this regard, we stress that there is a fundamental distinction between the cases $m=1$ and $m>1$.
When $m=1$, once a domain $\Omega$ satisfying (\ref{condI}) is constructed so that the integrals in (\ref{condII}) are finite for $|\alpha| \leq 1$, one can readily rescale the domain by different factors in each coordinate direction $z_j$ to enforce the moment condition in (\ref{condII}); see Remark~\ref{rmkmequal1case}. Thus, for $m=1$, condition~(\ref{condII}) introduces no additional difficulty beyond achieving (\ref{condI}).

In contrast, for larger values of $m$, the moment condition in (\ref{condII}) becomes highly nonlinear in the variables $|z_j|^2$ and coupled across different multi-indices $\alpha$, making the problem much more involved than condition~(\ref{condI}) alone. As a result, a direct construction appears out of reach, and a different approach is required. Our strategy is to reduce the problem to a mapping problem satisfying an appropriate nondegeneracy condition. This nondegeneracy ultimately enables us to invoke a Brouwer fixed point argument to establish the existence of the desired domains. To carry out this reduction, we perform a finite-dimensional cone analysis to show that the moment data in (\ref{condII}) admit a finite discrete representation satisfying an open cone condition. Another key idea is to treat the constant $c$ in (\ref{condII}) as a free parameter. Indeed, the nondegeneracy of the mapping problem depends crucially on several integral estimates, and the corresponding arguments appear to fail if $c$ is fixed in advance. To the best of our knowledge, our approach differs from existing methods in the literature.  For the convenience of the readers, we provide a more detailed explanation of our strategy at the beginning of \S\ref{sec3}, before we carry out the detailed construction.

Since the pioneering work of Wiegerinck~\cite{Wi}, it is known that for every positive integer $k$, there exists a domain whose Bergman space has prescribed dimension $k.$
On the other hand, the problem of prescribing the Bergman kernel has long been an appealing and more challenging question: Which positive real-analytic functions, and in particular which polynomials $P$, can arise as the Bergman kernel of a domain? 
Being a Bergman kernel imposes several necessary positivity conditions on $P$ (see, for instance, D'Angelo~\cite{DA25} for a clear exposition of various notions of positivity for polynomials). However, a precise characterization of which polynomials arise as Bergman kernels of domains remains far from being understood.
 We believe that our method opens a new approach for addressing such problems.

The rest of the paper is organized as follows. In \S\ref{sec2}, we establish the basic setup and develop analytic tools that will be used in the construction. The detailed construction of the domains in Theorem~\ref{thm1} is carried out in \S\ref{sec3}. For the convenience of the reader, we first explain the overall strategy of the construction at the beginning of \S\ref{sec3}. At the end of that section, in \S\ref{sec34}, we verify that the constructed domains have the desired properties, thereby completing the proof of Theorem~\ref{thm1}. Finally, in \S\ref{sec4}, we prove the statements made in various remarks in \S\ref{sec1}, and then establish Theorems~\ref{thm2}, \ref{thm20}, and \ref{thm3}.

\medskip

{\bf Acknowledgments.} The last author would like to thank Peter Ebenfelt and John Treuer for stimulating discussions and helpful comments. He is also grateful to Xiaojun Huang and Ngaiming Mok for many valuable insights on related topics over the years. The first three authors express their sincere gratitude to the last author, Professor Ming Xiao, for introducing them to this subject and for his invaluable and patient guidance throughout the completion of this project.

\begin{comment}
Throughout this note, the small positive parameter
\[
c>0
\]
is fixed at the beginning of the geometric construction. Every subsequent auxiliary parameter, such as
\[
\varepsilon_c,\ \delta_c,\ R(c),\ \text{and the bridge thickness},
\]
is allowed to depend on this fixed $c$. This is the logically clean way to organize the proof.

\bigskip

The goal is to explain a construction, for $n\ge 2$, of a connected Reinhardt domain $\Omega\subset \C^n$
whose Bergman kernel has the form
\[
K_\Omega(z,w)=C(1+z\cdot \overline w)^m
\]
for some $C>0$ and some fixed integer $m\ge 0$.
\end{comment}

\section{Analytic and geometric preliminaries for the construction}\label{sec2}

\subsection{Basic setup and notations}\label{sec21} %and ideas of the construction}\label{sec21}

\begin{definition}
Let $D$ be a domain in $\R^n$. We define the logarithmic lift of $D$, denoted by $\mathcal{L}(D)$, to be the domain in $\C^n$ given by
$$\mathcal{L}(D):=\{z=(z_1, \cdots, z_n) \in \C^n:\ (\log|z_1|,\cdots,\log|z_n|)\in D\}.$$
We call $D$ the logarithmic base of $\mathcal{L}(D)$. Note that $\mathcal{L}(D)$ is a Reinhardt domain and does not intersect any coordinate hyperplane $\{z_j=0\}$, $1 \le j \le n$.
\end{definition}

The following simple proposition reduces the integral of  squared moduli of monomials on $\mathcal{L}(D)$ to an integral on $D$. %For a multi-index $\alpha=(\al_1, \cdots, \al_n) \in N^n$, write $z^{\al}=z_1^{\al_1} \cdots z_n^{\al_n}.$ Let $dm_{2n}$ denote the Lebesgue measure element on $\C^n =\R^{2n}.$ 
Write $\langle \cdot~,~\cdot \rangle$ for the standard inner product in $\R^n.$

\begin{proposition}\label{prop:intloglift}
For every multi-index $\alpha=(\alpha_1,\cdots,\alpha_n) \in \mathbb{Z}^n,$ it holds that

\[
\int_{\mathcal{L}(D)}|z^\alpha|^2\,\dd V_{2n}(z)
=
(2\pi)^n
\int_{D} e^{2\langle \alpha+\mathbf{1}, x \rangle}\,\dd x.
\]
Here $ \mathbf{1}= (1, \cdots, 1)$, so that $\alpha + \mathbf{1} = (\alpha_1+1, \cdots, \alpha_n+1),$ and $\langle \alpha+\mathbf{1}, x \rangle=\sum_{j=1}^n (\alpha_j+1) x_j.$

\end{proposition}

\begin{proof}
To simplify the integral on the left hand side, we write each $z_j$ in the polar coordinates $z_j=r_j e^{i\theta_j}, 1 \le j \le n,$ 
with $r_j = |z_j|>0$. Then
%\[
%|z^\alpha|^2 = \prod_{j=1}^n r_j^{2\alpha_j},
%\qquad
%dV = \prod_{j=1}^n (r_j\,dr_j\,d\theta_j),
%\]
%so
	\[
	|z^\alpha|^2 \dd V_{2n}(z) = \prod_{j=1}^n r_j^{2\alpha_j+1} \, dr_j \, d\theta_j.
	\]	
Further set $x_j = \log r_j$, so $r_j = e^{x_j}$ and $dr_j = e^{x_j} dx_j$. Then we get
%\[
%r_j^{2\alpha_j+1} dr_j = e^{2(\alpha_j+1)x_j} dx_j,
%\]
%hence
	\[
	|z^\alpha|^2 \dd V_{2n}(z) = e^{2\langle \alpha+\mathbf{1}, x\rangle} \, dx \, d\theta.
	\]
Using the above equation and since $\mathcal{L}(D)$ is Reinhardt, we obtain
\[
	\int_{\mathcal{L}(D)} |z^\alpha|^2 \dd V_{2n}
	=
	\int_D \int_{[0,2\pi)^n} e^{2\langle \alpha+\mathbf{1}, x \rangle} d\theta \, dx
	=
	(2\pi)^n \int_D e^{2\langle \alpha+\mathbf{1}, x \rangle} dx.
\]
\end{proof}

We will construct a domain $D$ in $\R^n$ whose logarithmic lift $\Omega =\mathcal{L}(D)$ in $\C^n$ will give the desired domain in Theorem \ref{thm1}. Then by Proposition \ref{prop:intloglift}, condition (\ref{condII}) in \S\ref{sec1} is reduced to the following integral equation:

\begin{equation}
\int_{D} e^{2\langle \alpha+\mathbf{1}, x \rangle}\,\dd x=c A_\alpha:=\frac{c}{(2\pi)^n} \,\frac{\alpha!\,(m-|\alpha|)!}{m!}
\quad \text{for all } \alpha \in \N^n~\text{with}~|\alpha| \leq m.
\end{equation}

%Our desired domain $\Omega$ in Theorem will be  the logarithmic lift of some domain $D \subseteq \R^n.$ Thus, it amount to construct a domain $D$ in $\R^n$ which gives the desired properties of $\om$. 

%Our strategy to construct $D$ can be briefly summarized as follows.

%
%\begin{enumerate}[label=(\roman*)]
%\item We begin with the tail, since it is logically part of the fixed background geometry. First construct a tail in logarithmic coordinates whose contribution to all moments of degree $\le m$ can be made arbitrarily small, but whose contribution to every moment of degree $\ge m+1$ is infinite.
%\item Then construct a bounded connected logarithmic core whose low-degree moments equal a prescribed target vector \emph{minus} the small low-degree contribution of the tail.
%\item Finally attach the core to the tail by a very thin bridge, chosen \emph{before} the local correction step and therefore absorbed into the fixed part of the geometry.
%\end{enumerate}

%We fix some notations which will be used throughout the paper.

Therefore the function $e^{2\langle \alpha+\mathbf{1}, x \rangle}$ and the numbers $A_\alpha, |\alpha| \leq m,$ will be important and appear frequently in the later argument. For convenience, throughout the rest of the paper, we will use the following terminology. Fix a  pair of positive integers $(m,n)$ with $n \geq 2,$ as in Theorem \ref{thm1}.
\begin{itemize}
	\item As above,  $\N$ denotes the set of nonnegative integers, and $\Z_+$ denotes the set of positive integers. For a multi-index $\alpha=(\al_1, \cdots, \al_n) \in \N^n$, $|\alpha|:=\sum_{j=1}^n \alpha_j.$
	\item Set $\I:=\{\alpha\in \N^n:\ |\alpha|\le m\},$ and $N:=\#\I >n$ denotes the number of elements in $\I.$
	\item Write $A_\alpha:=\frac{1}{(2\pi)^n} \,\frac{\alpha!\,(m-|\alpha|)!}{m!}$ and $A=(A_\alpha)_{\alpha\in \I} \in \R^N.$
	\item  Write $\mathbf{1}:=(1,\cdots,1)$. For $\alpha\in\I$, set $v_{\alpha}(x):=e^{2\langle \alpha+\mathbf{1},x\rangle}$ for $x\in\R^n$, and define $v(x):=\bigl(v_{\alpha}(x)\bigr)_{\alpha\in\I}$, which is a map from $\R^n$ to $\R^N$.
	\item Set $K_v:=\operatorname{cone}\{v(x):x\in \R^n\}\subseteq \R^N$ to be the cone generated by $\{v(x) : x \in \R^n\}.$ That is,
	$$K_v=\left\{\sum_{i=1}^k a_i v(x_i) :  k \in \Z, k \geq 1, a_i \geq 0, x_i \in \R^n \right\}.$$
	\item For a bounded region $U \subseteq \R^k$,  denote by $\mathrm{Vol}_{k}(U)$ the volume of $U$ with respect to the Lebesgue measure in $\R^k.$ 
	\item Write $|\cdot|$ for the norm on $\R$, and $\|\cdot\|$ for the norm on $\R^k$ for $k \geq 2$.
	\item  Let  $f=(f_1, \cdots, f_N): U \subseteq \R^n \to \R^N$ be a map with each component $f_i$ integrable  on some open set  $ U \subseteq \R^n.$ Then
	we write
	$$\int_{U}f(x)\,\dd x:=\left(\int_{U}f_1(x)\,\dd x, \cdots, \int_{U}f_{N}(x)\,\dd x\right) \in \R^N.$$
	Note it always holds that $\|\int_{U}f(x)\,\dd x\| \leq \int_{U}\|f(x)\|\,\dd x.$
\end{itemize}

\begin{comment}
Write $\N$ for the set of nonnegative integers.
Let
\[
\I=\{\alpha\in \N^n:\ |\alpha|\le m\},
\qquad
N:=\#\I >n .
\]
For $x\in \R^n$ define
\begin{equation}\label{eqndefv}
	v(x):=\bigl(e^{2\langle \alpha+\mathbf{1}, x \rangle}\bigr)_{\alpha\in \I}\in \R^N.
\end{equation}
Let
\[
K:=\operatorname{cone}\{v(x):x\in \R^n\}\subset \R^N.
\]

We shall use the target vector
\[
A=(A_\alpha)_{\alpha\in \I},\qquad
A_\alpha:=\frac{1}{2^n(m+n)!}\,\alpha!\,(m-|\alpha|)!.
\]
\end{comment}

\subsection{Brouwer's fixed point type theorem}
In this subsection, we establish the following Brouwer's fixed point type theorem, which will be important for our later construction and proof.
\begin{proposition}\label{prop:brouwer}
	%[Quantitative local solvability on a slice]\label{lem:quantitative}
	Let $r$ be a positive real number and $A^0\in \R^N.$
	Denote by $\mathcal{B}_N(r)$ the ball in $\R^N$ centered at $0$ with radius $r$.
	Let $L:\R^N \to \R^N$ be an invertible linear map, and let
	\[
	G: \Ol{\mathcal{B}_N(r)} \to \R^N
	\]
be a continuous map which is $C^1$ in $\mathcal{B}_N(r)$. Denote by $\|\cdot\|_{\mathrm{op}}$ the operator norm of a linear operator between two normed spaces; and for each $\tau \in \mathcal{B}_N(r)$, write $DG(\tau)$ for the Jacobian matrix of $G$ at $\tau$, viewed as a linear map from $\R^N$ to $\R^N$.
Assume %there exists some constant $\lambda>0$ such that
	\[
	\lambda:=\sup_{ \tau \in \mathcal{B}_N(r)} \|DG(\tau)-L\|_{\mathrm{op}}\le \frac{1}{2\|L^{-1}\|_{\mathrm{op}}}.
	\]
	%\[
	%\lambda\|L^{-1}\| \le \frac12,
	%\]
	Assume also
	\[
	\|L^{-1}(G(0)-A^0)\|\le \frac r2.
	\]
	Then there exists $\tau^*\in \Ol{\mathcal{B}_N(r)}$ such that
	\[
	G(\tau^*)=A^0.
	\]
\end{proposition}

\begin{proof}
	Define a map $T$ from $\Ol{\mathcal{B}_N(r)}$ to $\R^N$ by 
	\[
	T(\tau):=\tau-L^{-1}(G(\tau)-A^0),~~\tau \in \Ol{\mathcal{B}_N(r)}.
	\]
	%For $\tau\in B_\tau(0,r)$,
	Since $G$ and $L^{-1}$ are both continuous on $\Ol{\mathcal{B}_N(r)}$, so is $T$. We next prove
	
	\medskip
	
	{\bf Claim:} $T$ maps the closed ball $\Ol{\mathcal{B}_N(r)}$ into itself.
	
	\medskip
	
	{\bf Proof of Claim:} By the assumptions and the Fundamental Theorem of Calculus, we have for any $\tau \in \mathcal{B}_N(r),$

	\[
	G(\tau)-G(0)-L\tau
	=
	\int_0^1 (DG(u \tau)-L)\tau\,\dd u.
	\]
Note also
	
	\[
	G(\tau)-A^0
	=\left( G(0)-A^0 \right)+ L \tau +
	\left( G(\tau)-G(0)-L\tau \right).
	\]
Applying $L^{-1}$ to the above equation and combining it with the previous equation, we obtain
	\[
	L^{-1}(G(\tau)-A^0)
	=
	L^{-1}(G(0)-A^0)+\tau+L^{-1}\int_0^1(DG(u \tau)-L)\tau\,\dd u.
	\]
	This yields, by the definition of $T,$
	\[
	T(\tau)
	=
	-L^{-1}(G(0)-A^0)-L^{-1}\int_0^1(DG(u \tau)-L)\tau\,\dd u,~~~\forall \tau \in \mathcal{B}_N(r).
	\]
	Taking norms and then using the assumptions gives
	\[
	\|T(\tau)\|\le \|L^{-1}(G(0)-A^0)\|+ \lambda \|L^{-1}\|_{\mathrm{op}} \|\tau\| < \frac r2+\frac r2 =r, ~~~\forall \tau \in \mathcal{B}_N(r).
	\]
Thus by continuity, $T$ maps the closed ball $\Ol{\mathcal{B}_N(r)}$ into itself. This proves the claim. \qed

Since $T$ is a continuous map from $\Ol{\mathcal{B}_N(r)}$ to itself, by Brouwer's fixed point theorem, $T$ has a fixed point $\tau^*\in \Ol{\mathcal{B}_N(r)}$:
\[
\tau^*=T(\tau^*).
\]
	This implies $L^{-1}(G(\tau^*)-A^0)=0$, and thus $G(\tau^*)=A^0$.
\end{proof}

\subsection{Finite discrete representation of moment data with a nondegeneracy condition}

We adopt the notation fixed in \S\ref{sec21}. In particular, $A_\alpha$ and $A$ denote the prescribed moment data given there. In this subsection, we establish a finite positive discrete representation of this data. To this end, we first prove the following proposition.  %Let  $f=(f_1, \cdots, f_N): U \subseteq \R^n \to \R^N$ be a map with each component $f_i$ integrable  on some open set  $ U \subseteq \R^n.$ Then we write
%$$\int_{U}f(x)\,\dd x:=\left(\int_{U}f_1(x)\,\dd x, \cdots, \int_{U}f_{N}(x)\,\dd x\right) \in \R^N.$$
%Note it always holds that $\|\int_{U}f(x)\,\dd x\| \leq \int_{U}\|f(x)\|\,\dd x.$

\begin{comment}
	Write $\N$ for the set of nonnegative integers.
	Let
	\[
	\I=\{\alpha\in \N^n:\ |\alpha|\le m\},
	\qquad
	N:=\#\I >n .
	\]
	For $x\in \R^n$ define
	\begin{equation}\label{eqndefv}
		v(x):=\bigl(e^{2\langle \alpha+\mathbf{1}, x \rangle}\bigr)_{\alpha\in \I}\in \R^N.
	\end{equation}
	Let
	\[
	K:=\operatorname{cone}\{v(x):x\in \R^n\}\subset \R^N.
	\]
	
	We shall use the target vector
	\[
	A=(A_\alpha)_{\alpha\in \I},\qquad
	A_\alpha:=\frac{1}{2^n(m+n)!}\,\alpha!\,(m-|\alpha|)!.
	\]
	
\end{comment}

\begin{proposition}\label{prop:integral-representation}
	Define
	\[
	\rho(x):=\frac{(m+n)!}{\pi^n m!}\frac{1}{(1+e^{2x_1}+\cdots + e^{2x_n})^{m+n+1}},~~x \in \R^n.
	\]
	Then for each $\alpha\in \I,$ we have
	\[
	A_\alpha=\int_{\R^n}\rho(x)e^{2\langle \alpha+\mathbf{1}, x \rangle}\,\dd x.
	\]
	Equivalently, with $v$ as given in $\S$\ref{sec21},
\begin{equation}\label{eqnarhov}
	A=\int_{\R^n}\rho(x)\,v(x)\,\dd x.
\end{equation}
\end{proposition}

\begin{proof}
Apply the change of variables $x \to y=(y_1, \cdots, y_n)$ where $y_j=e^{2x_j},~j=1,\cdots,n.$
	Then
	\[
	\dd x=\frac{1}{2^n}\frac{\dd y}{y_1\cdots y_n},~~~\text{and}~~
	e^{2\langle \alpha+\mathbf{1}, x \rangle}\dd x
	=
	\frac{1}{2^n} y^\alpha\,\dd y.
	\]
Moreover,
	\begin{equation}\label{eqnintxtoy}
		\int_{\R^n}\rho(x)e^{2\langle \alpha+\mathbf{1}, x \rangle}\,\dd x
		=\frac{(m+n)!}{(2\pi)^n m!}
		\int_{(0,\infty)^n}\frac{y^\alpha}{(1+y_1+\cdots+y_n)^{m+n+1}}\,\dd y.
	\end{equation}
	To compute the integral on the right hand side, we apply the change of variables: $y \to (s, \tilde{u}):=(s, u_1, \cdots, u_{n-1}),$ where
	$$s=y_1+\cdots+y_n,~\text{and}~u_j=\frac{y_j}{s}, \quad j=1,\cdots,n-1.$$
	Then under the new coordinates, the region $(0,\infty)^n$ becomes $(0, \infty) \times \Delta_{n-1},$ where
	\[
	\Delta_{n-1}= \{( u_1, \cdots, u_{n-1}) \in \R^n: ~\text{every}~ u_j>0,\,~\text{and}~ u_1+\cdots+u_n<1\}.
	\]
	Moreover, computing the Jacobian of this transformation gives $dy=s^{\,n-1}\,ds\,d\tilde{u}$. Define a function $u_n$ of $(s, u_1, \cdots, u_{n-1})$ by
	$$u_n=u_n(s, \tilde{u}):=\frac{y_n}{s}=1-(u_1+\cdots+u_{n-1}).$$ 
	Write $u:=(\tilde{u}, u_{n-1})$ and $u^{\alpha}:=u_1^{\alpha_1}\cdots u_n^{\alpha_n}$ for $\alpha =(\alpha_1, \cdots, \alpha_n) \in \N^n.$ Then $y^\alpha=s^{|\alpha|}u^\alpha.$ Consequently,
	
	\begin{equation}\label{eqninty}
		\begin{aligned}
			\int_{(0,\infty)^n}
			\frac{y^\alpha}{(1+y_1+\cdots+y_n)^{m+n+1}}\,dy
			&=
			\int_0^\infty
			\int_{\Delta_{n-1}}
			\frac{s^{|\alpha|}u^\alpha}{(1+s)^{m+n+1}}
			s^{n-1}\,d\tilde{u} \,ds
			\\[6pt]
			&=
			\left(\int_{\Delta_{n-1}} u^\alpha\,d \tilde{u} \right)
			\left(\int_0^\infty
			\frac{s^{|\alpha|+n-1}}{(1+s)^{m+n+1}}\,ds
			\right).
		\end{aligned}
	\end{equation}
By one of the standard Beta function integral formulas,
	\begin{equation}\label{eqnbeta1}
		\int_0^\infty
		\frac{s^{|\alpha|+n-1}}{(1+s)^{m+n+1}}\,ds=
		B(|\alpha|+n,m-|\alpha|+1)
		=
		\frac{(|\alpha|+n-1)!(m-|\alpha|)!}{(m+n)!}.
	\end{equation}
We also recall the following multivariable version of  Beta function integral. Let $\beta_1, \cdots, \beta_{n}>0.$ Then
$$\int_{\Delta_{n-1}}
	u_1^{\beta_1-1}\cdots u_{n-1}^{\beta_{n-1}-1}
	(1-u_1-\cdots-u_{n-1})^{\beta_n-1}
	\,d \tilde{u}
	=
	\frac{\Gamma(\beta_1)\cdots\Gamma(\beta_n)}{\Gamma(\beta_1+\cdots+\beta_n)}.$$
In particular, this yields
	\begin{equation}\label{eqnbeta2}
		\int_{\Delta_{n-1}} u^\alpha\,d \tilde{u}=\frac{\alpha!}{(|\alpha|+n-1)!}.
	\end{equation}
Then the conclusion follows from \eqref{eqnintxtoy}, \eqref{eqninty}, \eqref{eqnbeta1} and \eqref{eqnbeta2}.
\end{proof}

\begin{proposition}\label{prop:A-in-interior}
Let $K_v$ be the cone generated by the image of $v$, as defined in the notation list in \S\ref{sec21}. Then the vector $A$ lies in the interior $\operatorname{int} K_v$ of $K_v$.
\end{proposition}

\begin{proof}
	%From Proposition~\ref{prop:integral-representation}, we know that
	%\[
	%A=\int_{\R^n}\rho(x)v(x)\,\dd x.
	%\]
We first claim that \eqref{eqnarhov} implies $A \in \overline{K_v}$. To see this, define $f(x):=\rho(x)v(x):\R^n \to \R^N$, where $N$ is as in the notation list in \S\ref{sec21}. %Then $f(x) \in K$ 
Fix $\varepsilon > 0$. Choose a sufficiently large $n-$cube $U = [-R,R]^n$ such that
\begin{equation}\label{eqnaintu}
	 \left\| A- \int_{U} f(x)\, dx \right\|=\left\| \int_{\mathbb{R}^n \setminus U} f(x)\, dx \right\| < \frac{\varepsilon}{2}.
\end{equation}
Since $f$ is continuous on the compact set $U$, it is uniformly continuous. Partition $U$ into finitely many small $n-$cubes $\{W_i\}_{i=1}^M$,  such that
	\[
	\|f(x) - f(y)\| \le \frac{\varepsilon}{2 |U|} \quad \text{for all } x, y \in W_i.
	\]
For simplicity, we write $|W_i|$ for the volume $\mathrm{Vol}_{n}(W_i)$ of $W_i$.	Choose $p_i \in W_i$ for every $1 \leq i \leq M.$ Then
\begin{equation}\label{eqnintuminussep}
	\left\| \int_U f(x)\, dx - \sum_{i=1}^M f(p_i)\, |W_i| \right\|
	\le \sum_{i=1}^M \int_{W_i} \|f(x) - f(p_i)\|\, dx
	\le \frac{\varepsilon}{2}.
\end{equation}
	Set
	\[
	S_{\varepsilon} := \sum_{i=1}^M f(p_i)\, |W_i|
	= \sum_{i=1}^M \rho(p_i)\, |W_i|\, v(p_i).
	\]
Since $\rho(p_i)\,|W_i| > 0$, we have $S_{\varepsilon} \in K_v$.
But \eqref{eqnaintu} and \eqref{eqnintuminussep} yield  
$$\|A - S_{\varepsilon}\|  < \frac{\varepsilon}{2}+\frac{\varepsilon}{2}= \varepsilon.$$
Since $\varepsilon$ can be arbitrarily small, we conclude $A \in \overline{K_v}$.

	%because integrals of positive combinations lie in the closure of the cone.
	
To prove $A \in \operatorname{int} K_v,$ assume for contradiction that $A\notin \operatorname{int} K_v$. Since $K_v$ is a convex cone in
$\R^N$, there exists a nonzero linear functional $\Lambda:\R^N\to \R$ such that
	\[
	\Lambda(w)\ge 0~~~\text{for all }w\in K_v,~\text{and}~\Lambda(A)=0.
	\]
	Write
	\[
	\Lambda(u)=\sum_{\alpha\in \I} c_\alpha u_\alpha, \quad ~\text{where}~u=(u_\alpha)_{\alpha \in \I}.
	\]
(Here we use that $K_v$ is a cone, so the supporting hyperplane may be chosen to pass through the origin; hence $\Lambda$ may be taken to be linear, with no constant term.) Therefore, for every $x\in\R^n$, since $v(x)\in K_v$,
	\[
	\phi(x):=\Lambda(v(x))=\sum_{\alpha\in \I} c_\alpha e^{2\inner{\alpha+\mathbf{1}}{x}}\ge 0.
	\]
On the other hand, $\Lambda(A)=0$, together with \eqref{eqnarhov}, implies that
	\[\int_{\R^n}\rho(x)\phi(x)\,\dd x=0.
	\]
Since $\rho(x)>0$ and $\phi(x)\geq 0$ on $\mathbb{R}^n$, the above identity implies that $\phi(x)\equiv 0$ on $\mathbb{R}^n$. Consequently, writing $y_j=e^{2x_j}>0$, we obtain
\[
\sum_{\alpha\in \mathcal{I}} c_\alpha y^{\alpha+\mathbf{1}}=0,
\qquad \text{and hence} \qquad
\sum_{\alpha\in \mathcal{I}} c_\alpha y^\alpha=0
\]
for all $y\in (0,\infty)^n$. By analyticity, the latter identity holds for all $y\in \mathbb{R}^n$. It follows that all coefficients $c_\alpha$ vanish, contradicting the assumption that $\Lambda\neq 0$. Hence $A \in \operatorname{int} K_v$.
\end{proof}

Before we prove the next proposition, we recall the notion of a cone generated by a set  $S\subseteq \R^N:$

$$\operatorname{cone}(S):=\left\{\sum_{i=1}^k a_i y_i :  k \in \Z_+, a_i \geq 0, y_i \in S \right\}.$$

\begin{proposition}\label{prop:finite-positive}
Let $S \subseteq \mathbb{R}^N$ and let $K_S = \operatorname{cone}(S)$. If $B \in \operatorname{int} K_S$, then there exist a positive integer $M$, vectors $y_1, \cdots, y_M \in S$, and positive coefficients $a_1, \cdots, a_M > 0$ such that
	$$B=\sum_{j=1}^M a_j y_j,$$
	and such that the vectors $y_1,\cdots,y_M$ span $\R^N$. (In particular, $M \geq N$).
\end{proposition}

\begin{proof}
	Since $B \in \operatorname{int}K_S$, there exists $r>0$ such that $\{x \in \R^N: |x-B|< r \}\subseteq K_S.$
	Choose a basis $\xi_1,\cdots,\xi_N$ of $\R^N$ with all $|\xi_i|<r$. Then
	\[
	\Gamma:=\{B,\ B+\xi_1,\cdots, B+\xi_N,\ B-\xi_1,\cdots, B-\xi_N\} \subseteq K_S.
	\]

By the definition of the cone $K_S$, each  point in $\Gamma$ can be represented as a finite nonnegative combination of
	elements of $S$. Collect all elements of $S$ appearing in all these representations into a finite set
	$Y:=\{y_1,\cdots,y_M\}\subseteq S.$
	Then $B$ and all $B \pm \xi_i$ belong to $K_Y:=\operatorname{cone}(Y)$.
	Therefore
	\[
	\xi_i=\frac12\bigl((B+\xi_i)-(B-\xi_i)\bigr)\in \operatorname{span}(Y),
	\]
Consequently, $\operatorname{span}(Y)=\R^N$. Note that $B$ lies in the interior of the cone $K_Y$, i.e., $B \in \operatorname{int}K_Y$. (This can be seen from the facts that all $B \pm \xi_i$ belong to $K_Y$ and that $K_Y$ is convex.) Now set
	\[
	u:=y_1+\cdots+y_M.
	\]
	%Since $A\in \operatorname{int}K_Y$, 
	Then for sufficiently small $t>0$ we have $B-tu\in K_Y.$
	Thus
	\[
	B-tu=\sum_{j=1}^M b_j y_j~~~\text{for some}~~~b_j\ge 0.
	\]
	Consequently, $B=\sum_{j=1}^M (b_j+t)y_j,$
	and every coefficient $a_j:=b_j+t$ is strictly positive.
\end{proof}

By Proposition~\ref{prop:A-in-interior}, $A \in \operatorname{int} K_v.$ Applying Proposition~\ref{prop:finite-positive} to $S=\{v(x):x\in \R^n\}$ and $B=A$, we obtain:

\begin{corollary}\label{cor:discrete-A}
There exist points $p_1,\cdots,p_M\in \R^n$
and coefficients $a_1,\cdots,a_M>0$ for some positive integer $M$
such that
	\begin{equation}\label{eqnaincone}
		A=\sum_{j=1}^M a_jv(p_j),
	\end{equation}
	and the vectors $v(p_1),\cdots,v(p_M)$ span $\R^N$ (in particular, $M \geq N$).
By relabeling if necessary, we may and shall assume that the first $N$ vectors are linearly independent. Consequently, the matrix
	\begin{equation}\label{eqndefl}
		L:=\bigl(v(p_1),\ \cdots\ ,v(p_N)\bigr)\in \R^{N\times N}~\text{is invertible}.
	\end{equation}
Here the vectors in $\mathbb{R}^N$, in particular $v(p_j)$, are regarded as column vectors.
\end{corollary}

%\begin{remark}
%The finite discrete representation of the moment data in \eqref{eqnaincone}, and especially the nondegeneracy of $L$, will play a fundamental role in our later proof of using the Brouwer's fixed point  type theorem, Proposition \ref{prop:brouwer}.
%\end{remark}

\begin{remark}
It is clear that the choices of the integer $M$, the points $\{p_j\}_{j=1}^M$, and the coefficients $\{a_j\}_{j=1}^M$ depend only on the map $v:\R^n \to \R^N$ and the data $A \in \R^N$. The latter depend only on $m$ and $n$. Therefore,  $M$, $\{p_j\}_{j=1}^M$, and $\{a_j\}_{j=1}^M$ (and hence also $L$) depend only on $m$ and $n$.
\end{remark}

%%%%%%%%%%%%%%%%%%%%%%%%%%%%%%%%%%%%%%%%%%%%%%%%%%%%%%%%%%%%%%%%%%%%%%%%%%%%%%%%%%%%%%%%%%%%%%%%%%%%%%

\subsection{Preparation for the tail domain construction}
Let $(m,n)$ be a fixed  pair of positive integers with $n \geq 2$ as above. Set
\begin{equation}\label{eqndefgamma}
\gamma=\gamma_{m, n}:=\frac{2(m+n+1)}{n-1}.
\end{equation}
For  a parameter $R>1,$ define a domain in  $\R^n$ by
\[
T_R
:=
\left\{
x=(x_1, \cdots, x_n)\in \R^n~~\big|~~
t:=\frac{x_1+\cdots+x_n}{n}>R,
\quad
|x_j-t|<e^{-\gamma t},
\quad
j=1,\cdots,n
\right\}.
\]

\begin{remark}\label{rmk1taildomain}
Let $x \in T_R.$ Since for each $j$ we have $t - x_j \le |x_j - t| < e^{-\gamma t}$, it follows that $x_j > t - e^{-\gamma t} > R - 1$. Therefore, $T_R \subseteq (0,\infty)^n$. Moreover, for every $R > 1$,
\[
\{(x_1,\cdots,x_n) : x_1 = \cdots = x_n \in (R,\infty)\} \subseteq T_R.
\]
Consequently, for any $R, R' > 1$, we have $T_R \cap T_{R'} \neq \emptyset$. These facts will be used to prove part~(2) of Theorem~\ref{thm2}.
\end{remark}
%Geometrically, this is a very thin tube around the diagonal direction
%\[
%x_1=\cdots=x_n=t,
%\]
%whose cross-section shrinks exponentially as $t\to+\infty$.

\begin{proposition}\label{prop:tail}
%Fix $\gamma>0.$ Then 
For every multi-index $\alpha\in \N^n$, there exists  positive constants $C_{\alpha}$ and $\widetilde{C}_{\alpha},$ depending only on $\alpha,$ such that
\begin{equation}\label{eqnestimateinttr}
\widetilde{C}_{\alpha} \int_R^\infty e^{2(|\alpha|-m-1)u}\,\dd u < 
\int_{T_R} e^{2\langle \alpha+\mathbf{1}, x \rangle}\,\dd x
<  C_{\alpha}
\int_R^\infty e^{2(|\alpha|-m-1)u}\,\dd u.
\end{equation}
Consequently,

\begin{enumerate}[label=(\roman*)]
\item If $|\alpha|\le m$, i.e., $\alpha \in \I,$ then
\[
\int_{T_R} e^{2\langle \alpha+\mathbf{1}, x \rangle}\,\dd x<\infty,
\]
and furthermore this integral tends to $0$ as $R\to\infty$;

\item If $|\alpha|\ge m+1$, then
\[
\int_{T_R} e^{2\langle \alpha+\mathbf{1}, x \rangle}\,\dd x=+\infty.
\]
\end{enumerate}
\end{proposition}

\begin{proof}

As in the definition of $T_R$, set
$t:=\frac{x_1+\cdots+x_n}{n}.$ We also set $y_j:=x_j-t,$ and $y:=(y_1, \cdots, y_n), 1 \leq j \leq n.$ Then
\[
x_j=t+y_j,
\qquad
\sum_{j=1}^n y_j=0.
\]
Consequently, $x=t \mathbf{1}+y.$ We apply the following change of variables:
$$x=(x_1, \cdots, x_n) \to (t, \tilde{y}):=(t, y_1, \cdots, y_{n-1}),$$
and we regard $y_n$ as a function of $\tilde{y}:$
$$y_n=y_{n}(\tilde{y})=-\sum_{j=1}^{n-1} y_j.$$
By computing the Jacobian of the change of variables and its determinant, we see the volume element changes as follows: $dx=n\,dt\,d\tilde{y}.$  Furthermore, in the new coordinates, the domain $T_R$ becomes 
$$\{(t, \tilde{y}) \in \R^n ~|~ t \in (R, \infty), \tilde{y} \in E_t\},$$ 
$$~\text{where}~E_t=\left\{ \tilde{y} \in \R^{n-1}:\ |y_1|<e^{-\gamma t},~ \cdots,~|y_{n-1}|<e^{-\gamma t},~~~|y_n|=\left|\sum_{j=1}^{n-1} y_j\right|<e^{-\gamma t}  \right\}. $$ 
%the Lebesgue measure \(dx\) differs from \(dt\,d\sigma(y)\) only by a positive constant, where \(d\sigma(y)\) denotes the \((n-1)\)-dimensional Euclidean volume element on the hyperplane \(H\). Hence there exists a constant \(C_n>0\) such that
%\[
%dx=C_n\,dt\,d\sigma(y).
%\]
Therefore, writing $y=(\tilde{y}, y_n),$ we have
\begin{equation}\label{eqninttr1}
\int_{T_R}e^{2\langle \alpha+\mathbf{1}, x \rangle}\,dx
=n
\int_R^\infty
\int_{E_t}
e^{2 \langle\alpha+\mathbf{1},t\mathbf{1}+y\rangle}
\, d\tilde{y}\,dt,
\end{equation}
%where
%\[\mathbf{1}=(1,\cdots,1),\qquadE_t=
%\left\{
%y\in H:\ |y_j|<e^{-\gamma t}\right\}.
%\]
Next we separate the exponential factor:
\[
\langle \alpha+\mathbf{1},t\mathbf{1}+y\rangle
=
\langle \alpha+\mathbf{1},t\mathbf{1} \rangle +\langle \alpha+\mathbf{1}, y \rangle
=
(|\alpha|+n)t+\langle \alpha+\mathbf{1}, y \rangle.
\]
This helps us to further reduce the integral in \eqref{eqninttr1} into
\begin{equation}\label{eqninttr2}
\int_{T_R}e^{2\langle \alpha+\mathbf{1}, x \rangle}\,dx
=
n
\int_R^\infty
e^{2(|\alpha|+n)t}
\left(
\int_{E_t}e^{2\langle \alpha+\mathbf{1}, y \rangle}\,d \tilde{y}
\right)
dt.
\end{equation}
To estimate the inner integral over $E_t$  on the right hand side, we prove

\smallskip

{\bf Claim:} Let $t>0.$ The following two statements hold:

(a) Let $\alpha\in \N^n$, and set $b_{\alpha}:=2(|\alpha|+n)$. Then 

\[
e^{-b_{\alpha}}
<
e^{2\langle \alpha+\mathbf{1}, y \rangle}
<
e^{b_{\alpha}}, ~~~\forall  y=(\tilde{y}, y_n)~\text{with}~\tilde{y} \in E_t.
\]

\smallskip

(b) There exists a constant $V_0>0$ (which is in particular independent of $t$) such that the volume of $E_t$ in $\R^{n-1}$ satisfies
$\mathrm{Vol_{n-1}}(E_t)=V_0  e^{-\gamma (n-1)t}.$

\smallskip

{\bf Proof of Claim:}  To prove (a), we note for every $y=(\tilde{y}, y_n)$ with $\tilde{y} \in E_t,$

\[
|\langle \alpha+\mathbf{1}, y \rangle|
\le
\sum_{j=1}^n(\alpha_j+1)|y_j|
\le
(|\alpha|+n)e^{-\gamma t}.
\]
Consequently, with $b_{\alpha}=2(|\alpha|+n),$ we have
\[
-b_{\alpha} < -b_{\alpha}e^{-\gamma t}
<
2\langle \alpha+\mathbf{1}, y \rangle
\le
b_{\alpha}e^{-\gamma t}  \le b_{\alpha}.
\]
This yields the conclusion in (a). To prove (b), we set

$$E_0:= \left\{\tilde{u}=(u_1, \cdots, u_{n-1}) \in \R^{n-1}:\ |u_1|<1,~ \cdots,~|u_{n-1}|<1,~~~\left|\sum_{j=1}^{n-1} u_j\right|<1  \right\}.$$
We also write $V_0=\mathrm{Vol_{n-1}}(E_0).$  Now to compute the volume of $E_t$, we apply the change of variables: $(y_1, \cdots, y_{n-1}) \to (u_1, \cdots, u_{n-1})$ with $u_j=e^{\gamma t}y_j$, and get

$$\mathrm{Vol_{n-1}}(E_t)=\int_{E_t} d\tilde{y}= e^{-\gamma (n-1)t} \int_{E_0} d\tilde{u}=e^{-\gamma (n-1)t} V_0.$$
This proves part (b) and establishes the claim. \qed

\smallskip

We now return to the proof of Proposition~\ref{prop:tail}. By the claim, we immediately have

$$
e^{-b_{\alpha}}V_0 e^{-\gamma(n-1)t}
<
\int_{E_t} e^{2\langle \alpha+\mathbf{1}, y\rangle}\, d\tilde y
<
e^{b_{\alpha}}V_0 e^{-\gamma(n-1)t}.
$$

The above inequality together with \eqref{eqninttr2} yields, with
$C_{\alpha}:=e^{b_{\alpha}}nV_0$ and
$\widetilde C_{\alpha}:=e^{-b_{\alpha}}nV_0$,

$$
\widetilde C_{\alpha}\int_R^\infty
e^{2(|\alpha|+n)t}e^{-\gamma(n-1)t}\,dt
<
\int_{T_R} e^{2\langle \alpha+\mathbf{1}, x\rangle}\,dx
<
C_{\alpha}\int_R^\infty
e^{2(|\alpha|+n)t}e^{-\gamma(n-1)t}\,dt.
$$

By the definition of $\gamma$ in \eqref{eqndefgamma}, we arrive at \eqref{eqnestimateinttr}. Finally, assertions (i) and (ii) follow easily from \eqref{eqnestimateinttr}. This finishes the proof of Proposition~\ref{prop:tail}.
\end{proof}

\section{Construction of the desired domain and proof of Theorem \ref{thm1}}\label{sec3}

%\subsection{Something before construction}

In this section, we construct the desired domains and establish Theorem \ref{thm1}. Before we start the construction, we fix some data and notations which will be used throughout the construction.

Let the  positive integer $M$, the  points $p_1,\cdots,p_M\in \R^n$
and coefficients $a_1,\cdots,a_M>0, $ as well as the $N \times N$ matrix $L$,  be as in Corollary~\ref{cor:discrete-A}.
Choose two big balls $\mathcal{B}_1, \mathcal{B}_2$ centered at $0$ in $\R^n$ such that 
\[
\{p_1,\cdots,p_M\}  \subseteq  \mathcal{B}_1 \subseteq \Ol{\mathcal{B}_1} \subseteq  \mathcal{B}_2 \subseteq \Ol{\mathcal{B}_2}.
\]
For $r>0$ and $p \in \R^n$, we denote by $B(p,r)$ the ball in $\R^n$ centered at $p$ with radius $r$. Pick $\rho>0$ sufficiently small so that $\{B(p_j,2\rho)\}_{j=1}^M$ are
pairwise disjoint and precompact balls contained in $\mathcal{B}_1$.
%Note the choice of $Q_1, Q_2, Q_3$ and $\rho$ do not depend on $c.$
Write 
\begin{equation}\label{eqnka12}
\kappa_1=\sup_{x \in \Ol{\mathcal{B}_2}} \|v(x)\|, ~\text{and}~\kappa_2=\sup_{x \in \Ol{\mathcal{B}_2}} \|v'(x)\|.
\end{equation}
%\begin{equation}\label{eqnlm12}
%\kappa:=\sup_{x \in \Ol{\mathcal{B}_2}} \|v'(x)\|.
%\end{equation}
Here $v'$ denotes the map obtained from $v$ by taking the derivative component-wise. Furthermore, set
\begin{equation}\label{eqndefs}
S(x):=\sum_{j=1}^n e^{2x_j};
\end{equation}
%Define
\begin{equation}\label{eqns0}
	s:=\min_{x \in \Ol{\mathcal{B}_2}}|S(x)|>0, ~\text{and}~s_0:=\min \left\{\frac{s}{2}, \frac{1}{N}\right\}<s.
\end{equation}

%In this section we combine the bounded geometric skeleton, the bridge to the tail, the tail itself, and the variable branches into one family, and then apply Lemma~\ref{lem:quantitative}. Following the discussion, we \emph{fix} a small positive number $c$ at the beginning of the construction, and all later choices depend on this $c$.
\begin{comment}
\begin{theorem}[Geometric realization theorem]\label{thm:geometric-realization}
Assume $n\ge 2$, and let the discrete data
\[
A=\sum_{j=1}^M a_j\,v(p^{(j)}),\qquad a_j>0,
\]
be as in Corollary~\ref{cor:discrete-A}, with
\[
L=\bigl(v(p^{(1)})\ \cdots\ v(p^{(N)})\bigr)
\]
invertible. Then for every sufficiently small $c>0$ there exists  a connected open set
\[
D_{c}\subset \R^n
\]
such that:

\begin{enumerate}[label=(\roman*)]
\item for every $\alpha\in \I$,
\[
\int_{D_{c}} e^{2\inner{\alpha+1}{x}}\,\dd x
=
(cA)_\alpha;
\]

\item for every multi-index $\alpha$ with $|\alpha|\ge m+1$,
\[
\int_{D_{c}} e^{2\inner{\alpha+1}{x}}\,\dd x=+\infty.
\]
\end{enumerate}
\end{theorem}

\end{comment}

\begin{remark}
	Note that the choices of the above data—$\mathcal{B}_1, \mathcal{B}_2, \rho, \kappa_1, \kappa_2, s, s_0$—depend only on $\{p_j\}_{j=1}^M$ and the map $v(x)$, and hence ultimately depend only on $m$ and $n$.   
\end{remark}

Let $0< c<1$ be a parameter, which we will make very small in later proof.  We also fix some auxiliary parameters which depend on $c$.
Fix any  constant $0<\beta<\frac1{n-1}$ and set
\[
\varepsilon_c:=c^\beta, \qquad %\qquad 0<\beta<\frac1{n-1},
~\text{an $(n-1)$-cube}~~Q_c:=(-\varepsilon_c,\varepsilon_c)^{n-1} \subseteq \R^{n-1};
\]
\[
\sigma_c:=\mathrm{Vol}_{n-1}(Q_c)=(2\varepsilon_c)^{n-1}=2^{n-1}c^{\beta(n-1)},
\qquad
\delta_c:=c^2.
\]
Then it is clear that, as $c\to 0^+$, it holds that
\[
\varepsilon_c\to 0,\qquad \frac{c}{\sigma_c}\to 0,\qquad \delta_c \sigma_c=o(c^2).
\]

Before presenting the detailed construction of the desired domains, we first outline the general strategy. For each small $c>0$, we aim to construct a domain $D_c^*$ which, among other things, satisfies the following two conditions:

\begin{enumerate}
	\renewcommand{\labelenumi}{(\Alph{enumi})}
	\renewcommand{\theenumi}{(\Alph{enumi})}
	\item \label{condA} $\int_{D_c^*} e^{2\langle \alpha+\mathbf{1}, x \rangle}\,\dd x=\infty,~\forall \alpha\in \Z^n \setminus \I;$
	\item \label{condB} $\int_{D_c^*} e^{2\langle \alpha+\mathbf{1}, x \rangle}\,\dd x=c A_\alpha, ~\forall \alpha\in \I.$
\end{enumerate}

The main difficulty lies in fulfilling condition~\ref{condB}. To address this, for each small $c>0$, we first construct a family of domains $D_c(T)$, where $T$ is a parameter taking values in $\R^M$. All domains $D_c(T)$ satisfy condition~\ref{condA}. In general, they do not satisfy condition~\ref{condB}, but only approximate it. We then apply a Brouwer fixed point type argument to show that, for a suitable choice of $T$, the domain $D_c(T)$ satisfies condition~\ref{condB} exactly, thereby yielding the desired domain $D_c^*$.

The first key step is therefore to construct the family of candidate domains $D_c(T)$. These domains consist of several subdomains. To ensure that $D_c(T)$ is connected, we introduce a shell-shaped hub subdomain $O_c \subseteq D_c(T)$ such that every other subdomain is connected to $O_c$ by a thin corridor. The subdomains of $D_c(T)$ are listed below, and their detailed construction will be carried out in the next subsections \S\ref{sec31}--\ref{sec33}.

\smallskip

(1). The hub domain $O_c$ and two tail domains $T_c, \tilde{T}_c$, together with thin corridors connecting them to $O_c$, will be constructed in \S\ref{sec31}. The two tail domains $T_c$ and $\tilde{T}_c$ are used to ensure the condition in (A) for $\alpha\in \N^n \setminus \I$ and $\alpha \in \Z^n \setminus \N^n$, respectively. 

(2). The anchor domains $U_{j,c}$ and the variable domains $V_{j,c}$, $1 \leq j \leq M$, together with thin corridors connecting them to $O_c$, will be constructed in \S\ref{sec32}. 

The domains $U_{j,c}$ and $V_{j,c}$ are carefully chosen $n$-cubes based at points $p_j$, $1 \leq j \leq M$, where the points $\{p_j\}_{j=1}^M$ are given in Corollary \ref{cor:discrete-A}. We associate to each $V_{j,c}$ a real parameter $t_j \in \R$, writing $V_{j,c}=V_{j,c}(t_j)$, and thus $T=(t_1, \cdots, t_M)$ parametrizes the ultimate candidate domains $D_c(T)$ (all subdomains of $D_c(T)$ other than $V_{j,c}$ are independent of $T$).

The role of the anchor domain $U_{j,c}$ is to ensure a fixed intersection with $V_{j,c}$—namely, $U_{j,c}\cap V_{j,c}$ is independent of $T$. This guarantees that $V_{j,c}$ remains connected to $O_c$ and, more importantly, allows us to carry out the geometric analysis primarily on the difference $W_{j,c}:=V_{j,c} \setminus U_{j,c}$, which is crucial for the subsequent Brouwer fixed point type argument.

(3). The final domain $D_c(T)$ is defined as the union of the above subdomains, and thus depends on the parameter $T=(t_1, \cdots, t_M)$ through $V_{j,c}(t_j)$. This enables us to formulate a mapping problem in $\R^M$. Using \eqref{eqndefl}, we show that, after restricting to an $N$-plane in $\R^M$, the mapping satisfies a suitable nondegeneracy condition. We finally apply a Brouwer fixed point type argument to conclude that, for each small $c>0$, there exists $T_c$ such that $D_c(T_c)$ satisfies condition~\ref{condB}, and hence is the desired domain $D_c^*$. This is carried out in \S\ref{sec33}. In the final section \S\ref{sec34}, we show that the logarithmic lift of $D_c^*$ yields the desired domain in Theorem~\ref{thm1}.

\subsection{Shell-shaped hub domain and two tail domains}\label{sec31}

\textbf{Step 1. Construction of the shell-shaped hub domain $O_c$}

\smallskip

We first choose a sufficiently thin spherical shell $O_c$ centered at $0$; that is,
$O_c=\{x \in \R^n: r_{c,1}<\|x\|<r_{c,2}\}$
for some $r_{c,2}>r_{c,1}>0$, such that
\begin{equation}\label{eqnintoc}
	\overline{O_c} \subseteq \mathcal{B}_2 \setminus \overline{\mathcal{B}_1}, \quad \text{and} \quad
	\left\|\int_{O_c} v(x)\,\dd x\right\|\le c^2.
\end{equation}
Note the shell $O_c$ encloses the smaller ball $\mathcal{B}_1$ and lies inside the larger ball $\mathcal{B}_2$. The domain $O_c$ serves as a hub so that every other subdomain of $D_c(T)$ constructed below will be connected to $O_c$.
%\textbf{Step 2. Fix the large cube and the corner ball.}
%The latter is doable as $\mathcal{B}_2$ does not depend on $c$. For instance, we can take a larger ball $\mathcal{B}_3$ with $\Ol{\mathcal{B}_2} \subseteq \mathcal{B}_3,$ and let 
%$\lambda_1=\sup_{x \in \Ol{\mathcal{B}_3}} \|v(x)\|$. Then we just need to make $O_c$ thin enough so that its volume satisfies
%$$\mathrm{Vol_{n}}(O_c) \leq \frac{c^2}{\lambda_1}.$$
%containing all the points $p^{(j)}$, and 

\medskip

\noindent
\textbf{Step 2. Construction of the first tail domain and corridor connecting it to the hub}

\medskip

Recall by Proposition~\ref{prop:tail}, for any multi-index $\alpha\in \N^n$ and $R>1,$ % with $|\alpha|>m,$ 

\begin{equation}\label{eqninttrc1}
\int_{T_R} e^{2\langle \alpha+\mathbf{1}, x \rangle}\,\dd x = \infty,~~~\forall |\alpha|>m;
\end{equation}
%On the other hand, %for each fixed $\alpha\in \I$,
\begin{equation}\label{eqninttrc2}
\int_{T_R} e^{2\langle \alpha+\mathbf{1}, x \rangle}\,\dd x<\infty~~~\text{and}~~ \lim_{R\to \infty} \int_{T_R} e^{2\langle \alpha+\mathbf{1}, x \rangle}\,\dd x=0,~~~\forall \alpha\in \I.
\end{equation}
Since $T_R  \subseteq \{x_1+\cdots+x_n > nR\}$ and $\I$ is finite, we may choose  R to be a very large number $R_c$ (depending on $c$) such that
\begin{equation}\label{eqninttrc3}
\mathcal{B}_2  \cap T_{R_c}  =\emptyset, ~~\text{and}~~\left\|
\int_{T_{R_c}} v(x)\,\dd x
\right\|\le  \frac{c^2}{2}.
\end{equation}
%\[
%(2). \left\| \int_{D_{\mathrm{tail}}^{(R_c)}} v(x)\,\dd x\right\|\le  c^2
%\]
%From now on the tail parameter is fixed to be this $R(c)$.
Pick a line segment $l_c$ in $\R^n$ connecting a point $q_1 \in O_c$ and a point $q_2 \in T_{R_c}$ such that $l_c \cap \Ol{\mathcal{B}_1} =\emptyset.$ 
Let $J_c$ be a thin tubular neighborhood of $l_c$ so that $O_c \cup T_{R_c} \cup J_c$ is connected. Heuristically,  $J_c$ can be viewed as a very thin corridor connecting $T_{R_c}$ to $O_c.$
Making $J_c$ sufficiently thin, we further have
 %(in particular, $J_c$ contains some neighborhoods of $q$ and $q'$) such that the following hold:
%$$(1). O_c \cup J_c \cup D_{\mathrm{tail}}^{(R_c)}~\text{is connected};$$
\begin{equation}\label{eqninttrc4}
 J_c \cap \Ol{\mathcal{B}_1} =\emptyset, ~~\text{and}~~
\left\|
\int_{J_c} v(x)\,\dd x
\right\|\le  \frac{c^2}{2}.
\end{equation}

\begin{remark}\label{rmklcjc}
Since $O_c$ is a spherical shell and  $T_{R_c} \subseteq (0,\infty)^n$ (see Remark \ref{rmk1taildomain}), we can pick $l_c, J_c$ so that  they are both contained in $(0,\infty)^n$. This special choice will only be used to prove part (2) of Theorem \ref{thm2}.
\end{remark}
%Roughly speaking, $J_c$ is a very thin corridor connecting $B_c$ and $D_{\mathrm{tail}}^{(R_c)}$. 
%\[
%B_*\subset Q_{\mathrm{big}}
%\]
%near one corner of $Q_{\mathrm{big}}$.

\noindent
\textbf{Step 3. Construction of the second tail domain and corridor connecting it to the hub}

\medskip

Let $s$ and $s_0$ be as defined in \eqref{eqns0}. Fix any $0< \mu <1.$ Now for every $0<c<1,$ set

\begin{equation}\label{eqndeftildetc}
\widetilde{T}_{c}=\widetilde{T}_{c}(\mu)
:=
\left\{
x\in \mathbb{R}^n:
\mu c s_0
<
S(x):=e^{2x_1}+\cdots+e^{2x_n}
<
cs_0
\right\}.
\end{equation}

\begin{remark}
	We can fix $\mu$ for most of the subsequent arguments, and will treat it as varying in $(0,1)$ only when proving part~(2) of Theorem~\ref{thm2}. Accordingly, for simplicity we write $\widetilde{T}_c(\mu)$ as $\widetilde{T}_c$.
\end{remark}

Note by \eqref{eqns0},  we have $S(x) \geq s>cs_0$ on $\Ol{\mathcal{B}_2}$. Consequently, $\widetilde{T}_c \cap \Ol{\mathcal{B}_2}=\emptyset.$
For each $0<c<1,$ we also define
\[
G_{c}=G_{c}(\mu):=
\left\{
y\in (0,\infty)^n:
\mu c s_0 < y_1+\cdots+y_n < cs_0
\right\}.
\]
Note $\widetilde{T}_{c}$ becomes $G_{c}$ under the change of variables
$y_j = e^{2x_j},~1 \leq j \leq n.$ Note since $G_{c}$ is connected (it is indeed convex), so is $\widetilde{T}_c.$ Moreover, 
$$G_{c} \subseteq (0, cs_0)^n.~\text{Consequently},~\mathrm{Vol_{n}}(G_{c}) \leq (cs_0)^n < \frac{c^2}{N}.$$
Here we have used the fact that $n \geq 2$ and $s_0 \leq  \frac{1}{N}< 1$ to get the last inequality. 
We next prove

\begin{lemma}\label{lmtildetcint}
For a multi-index $\alpha \in \N^n$, we have
\begin{equation}\label{eqninttc1}
\int_{\widetilde{T}_{c}} e^{2\langle \alpha+\mathbf{1}, x \rangle}\,dx < \frac{c^2}{2^n N}.
\end{equation}
Consequently,
\begin{equation}\label{eqninttc2}
	\left\|\int_{\widetilde{T}_{c}} v(x) \,dx \right\| < \frac{c^2}{2^n \sqrt{N}} < \frac{c^2}{2}.
\end{equation}
\end{lemma}

\begin{proof}
To study the integral in \eqref{eqninttc1}, we apply the change of variables $y_j = e^{2x_j},~1 \leq j \leq n,$ to derive 
\begin{comment}
\[
dx = \frac{1}{2^n}\frac{dy}{y_1\cdots y_n},
\qquad
e^{2\langle \alpha+1,x\rangle}dx
=
\frac{1}{2^n} y^\alpha \,dy.
\]
\end{comment}
\[
\int_{\widetilde{T}_{c}} e^{2\langle \alpha+\mathbf{1}, x \rangle} \,dx
=
\frac{1}{2^n}
\int_{G_c} y^\alpha \,dy.
\]
Furthermore, for $y=(y_1, \cdots, y_n) \in G_{c},$ we have each $0< y_j< cs_0<1.$ Consequently,
$$y^\alpha \leq 1,~\text{and thus}~ \int_{G_{c}} y^\alpha \,dy \leq \mathrm{Vol_{n}}(G_{c}) < \frac{c^2}{N}.$$
Then \eqref{eqninttc1} follows immediately, and together with the definition of $v$ in $\S$\ref{sec21}, it yields \eqref{eqninttc2}.
\end{proof}

\begin{remark}\label{rmktildetc}
The above lemma shows that adding $\widetilde{T}_{c}$ to our domain under construction does not impact the integrability of $e^{2\langle \alpha+\mathbf{1}, x \rangle}$ for $\alpha \in \N^n$. On the other hand, by a very similar argument as in Lemma \ref{lmtildetcint}, one can show that $e^{2\langle \alpha+\mathbf{1}, x \rangle}$ is not integrable on $\widetilde{T}_{c}$  for $\alpha \in \Z^n \setminus \N^n $, so that condition \ref{condA}  is fulfilled for all $\alpha \in \Z^n \setminus \N^n $. But we will not directly use this property. It is indeed more convenient to use the spirit of this property on the logarithmic lift of our constructed domain. See the proof of Theorem \ref{thm1} in \S\ref{sec34}.
\end{remark}

Pick a path $\widetilde{l}_{c}$ in $\R^n$ connecting a point $\widetilde{q}_1 \in O_c$ and a point $\widetilde{q}_2 \in \widetilde{T}_{c}$ such that $\widetilde{l}_{c} \cap \Ol{\mathcal{B}_1} =\emptyset.$ 
Let $\widetilde{J}_c$ be a thin tubular neighborhood of $l_c$ so that $O_c \cup \widetilde{J}_c \cup \widetilde{T}_{c}$ is connected. %Heuristically,  $J_c$ can be viewed as a very thin corridor connecting $T_{R_c}$ to $O_c.$
Making $\widetilde{J}_c$ sufficiently thin, we further have
%(in particular, $J_c$ contains some neighborhoods of $q$ and $q'$) such that the following hold:
%$$(1). O_c \cup J_c \cup D_{\mathrm{tail}}^{(R_c)}~\text{is connected};$$
\begin{equation}\label{eqninttc3}
	\widetilde{J}_c \cap \Ol{\mathcal{B}_1} =\emptyset, ~~\text{and}~~
	\left\|
	\int_{\widetilde{J}_c} v(x)\,\dd x
	\right\|\le  \frac{c^2}{2}.
\end{equation}

\begin{remark}\label{rmktildelj}
Note since $\widetilde{q}_1 \in O_{c} \subseteq \mathcal{B}_2,$  by the definition \eqref{eqns0} of $s$ and $s_0$, $S(\widetilde{q}_1) \geq s> \mu cs_0.$ Since  $\widetilde{q}_2 \in \widetilde{T}_{c}$, by the definition of $\widetilde{T}_{c},$ $S(\widetilde{q}_2)> \mu cs_0.$  Since $\{x \in \R^n: a< S(x)< b\} $ is connected for any $b>a>0,$ we can carefully choose $\widetilde{l}_{c}$ and make $\widetilde{J}_c$ sufficiently thin, such that we additionally have
$\inf_{x \in \widetilde{J}_c} S(x)>\mu cs_0.$ Consequently,
$$\inf_{x \in \widetilde{T}_{c} \cup \widetilde{J}_c} S(x)=\mu cs_0.$$
This will only be used to prove part (2) of Theorem \ref{thm2}.
\end{remark}

\smallskip

\begin{remark}\label{rmkmequal1case}
If $m=1$, then we can stop here and no further constructions are needed. Indeed, taking  $D:=O_c \cup (T_{R_c} \cup J_c) \cup (\widetilde{J}_c \cup \widetilde{T}_{c}),$
the logarithmic lift $\Omega:=\mathcal{L}(D) \subseteq \C^n$, after a linear transformation, is the desired domain. To see this, by \eqref{eqninttrc1}, \eqref{eqninttrc2} and the discussion in Remark \ref{rmktildetc},  we note that
$e^{2\langle \alpha+\mathbf{1}, x \rangle}, \alpha \in \Z^n,$ is integrable on $D$ if and only if $\alpha \in \N^n$  and $|\alpha| \le m=1.$ Consequently, using Proposition \ref{prop:intloglift}, one can show the Bergman space of $\om$ is spanned by $\{1, z_1,\cdots, z_n\}.$ Since $\om$ is Reinhardt, its Bergman kernel must take the form 
$$K_{\om}=b_0 +b_1 |z_1|^2+\cdots +b_n |z_n|^2~~~\text{with all}~~b_i>0.$$
By the transformation law of the Bergman kernel, applying an appropriate linear transformation to $\om$, we can obtain a new domain in $\C^n$ whose Bergman kernel is given by
$$b(1+|z_1|^2+\cdots  +|z_n|^2)~~~\text{for some}~~b>0.$$ 
This is precisely the desired Bergman kernel for $m=1$. However, this argument does not work for the general case $m \ge 2$ due to the more complicated nature of the moment data condition (\ref{condII}) in \S\ref{sec1}. %and thus further constructions are needed.
\end{remark}

%\textbf{Step 4. Bounded skeleton inside the large cube.}

\subsection{The anchor domain and the variable domain}\label{sec32}

\textbf{Step 1. Construction of anchor domain  $U_{j,c}$ and corridor connecting it to the hub}

\smallskip

Inside each $B(p_j,\rho), 1 \leq j \leq M,$ we place an anchor box
\[
U_{j,c}:=p_j+\bigl((-4\delta_c,0)\times Q_c\bigr).
\]
The  volume of $U_{j,c}$ with respect to the Lebesgue measure in $\mathbb{R}^n$ is given by:
$$\mathrm{Vol}_{n}(U_{j,c})=4 \delta_c \cdot \mathrm{Vol}_{n-1}(Q_c)= 4 \delta_c\sigma_c=2^{n+1} c^2 c^{\beta(n-1)}.$$
Recall that the choices of $M$, $\rho$, and $\kappa_1$ depend only on $m$ and $n$. Therefore, we may choose a small constant $c_0\in(0,1)$, depending only on $m$ and $n$, such that for every $0<c<c_0$, the conditions in \eqref{eqnajcinb} and \eqref{eqncbeta} are satisfied.
\begin{equation}\label{eqnajcinb}
\Ol{U_{j,c}} \subseteq B(p_j,\rho),~\forall 1 \leq j \leq M, ~~~\text{thus in particular},~~ 4\delta_c < \rho;
\end{equation}
\begin{equation}\label{eqncbeta}
2^{n+1} c^{\beta(n-1)} \leq \frac{1}{2 M \kappa_1},~~\text{where}~\kappa_1~\text{is given in \eqref{eqnka12}}.
\end{equation}
Consequently, we have 
\begin{equation}\label{eqnujc}
\left\|
\int_{U_{j,c}} v(x)\,\dd x
\right\|\le \kappa_1 \mathrm{Vol_{n}}(U_{j,c})  \leq   \frac{c^2}{2M},~~~\forall 1 \leq j \leq M.
\end{equation}
Write for each $1 \leq j \leq M$,
$$\tilde{p}_j:=p_j+(-2\delta_c, 0, \cdots, 0) \in U_{j,c} ~~~~~\text{and}~~~~~~\hat{p}_j:=p_j+(-2\rho, 0, \cdots, 0) \in \partial B(p_j,2\rho).$$
Let $\alpha_{j, 1}$ be the following line segment connecting $\tilde{p}_j$ and $\hat{p}_j$:
$$\alpha_{j, 1}(t):= p_j+(-t, 0, \cdots, 0), t\in [2\delta_c, 2\rho].$$
Choose a path $\alpha_{j, 2}$ in $\mathcal{B}_2$ connecting $\hat{p}_j$ to a point in $O_c$ such that $\alpha_{j, 2}$ does not intersect  $B(p_k,2\rho)$ for any $1 \leq k \leq M.$
This is doable because, if $\alpha_{j, 2}$ intersects $B(p_k,2\rho)$ for some $k$, we can replace the portion inside $B(p_k,2\rho)$ with a curve lying on the boundary $\partial B(p_k,2\rho)$, thereby avoiding $B(p_k,2\rho)$. %Similarly, we can make $\beta_{j}$ avoid $B(p^{(j)},\rho)$.

Let $\alpha_j:=\alpha_{j, 1}+ \alpha_{j, 2}$ be the concatenation of $\alpha_{j, 1}$ and $\alpha_{j, 2}$. Then 
$\alpha_j$ is a path connecting $\tilde{p}_j$ to a point in $O_c$.
Let $\Gamma_{j, c} \subseteq \mathcal{B}_2$ be a thin tubular neighborhood of $\alpha_j$ such that the following conditions in \eqref{eqngammajc} and \eqref{eqngammabpj} hold for every $1 \leq j \leq M,$
\begin{equation}\label{eqngammajc}
\left\|
\int_{\Gamma_{j, c}} v(x)\,\dd x
\right\|\le  \frac{c^2}{2M}.
\end{equation}

\begin{equation}\label{eqngammabpj}
\text{If}~p_j+(x_1, x_2, \cdots, x_n) \in \Gamma_{j, c} \cap B(p_j,\rho),~\text{then}~x_1<-\delta_c.
\end{equation}
The latter is feasible because $\{\alpha_j(t)\} \cap B(p_j,\rho)=\{\alpha_{j,1}(t)\} \cap B(p_j,\rho)$ and because of the specific form of $\alpha_{j,1}$.

%Roughly speaking,  $\Gamma_{j, c}$ is a thin corridor connecting $A_{j,c}$ to $B_c$. 
%let
%\[
%B_{c,\mathrm{exit}}
%\]
%denote the thin corridor joining $B_*$ to the chosen exit point of the large cube.

%%%%%%%%%%%%%%%%%%%%%%%%%%%%%%%%%%%%%%%

\begin{comment}

Using Lemma~\ref{lem:avoidance}, choose pairwise disjoint thin polygonal corridors from $B_*$ to the branch
neighborhoods $B(p^{(j)},\rho/2)$, and one extra thin polygonal corridor from $B_*$ to a chosen exit point on a face
of $Q_{\mathrm{big}}$ that leads toward the tail region.

Inside $B(p^{(j)},\rho/2)$ place an anchor box
\[
A_{j,c}:=p^{(j)}+\bigl((-2\delta_c,0)\times U_c\bigr).
\]
Let $R_{j,c}$ denote the thin corridor joining $B_*$ to $A_{j,c}$, and let
\[
B_{c,\mathrm{exit}}
\]
denote the thin corridor joining $B_*$ to the chosen exit point of the large cube.

By taking all these corridors sufficiently thin, we may ensure that 

\end{comment}

%%%%%%%%%%%%%%%%%%%%%%%%%%%%%%%%%%%%

%\bigskip

Next define the following skeleton domain inside $\mathcal{B}_2:$
\[
S_c^{\mathrm{bd}}
:=
\bigcup_{j=1}^M (\Gamma_{j,c}\cup U_{j,c}).
\]
Note $S_c^{\mathrm{bd}}$ is an open subset $\mathcal{B}_2$, and $O_c \cup S_c^{\mathrm{bd}}$ is connected. Furthermore,  by \eqref{eqnujc} and \eqref{eqngammajc}, we have 

\begin{comment}

 and that the volume of $S_c^{\mathrm{bd}}$  satisfies
\[
\mathrm{Vol}(S_c^{\mathrm{bd}})\le \frac{c^2}{3M_0},
\]
where
\[
M_0:=\sup_{x\in Q_{\mathrm{big}}} |v(x)|.
\]
Consequently,
\end{comment}

\begin{equation}\label{eqnscbd}
\left\|\int_{S_c^{\mathrm{bd}}} v(x)\,\dd x\right\|\le M \left( \frac{c^2}{2M} + \frac{c^2}{2M} \right)=c^2.
\end{equation}
\begin{comment}
\textbf{Step 5. Bridge from the large cube to the tail.}

Connect the exit point of $Q_{\mathrm{big}}$ to the initial portion of the tail $D_{\mathrm{tail}}^{(R(c))}$ by an open
bridge
\[
B_{c}^{\mathrm{tail}}
\]
lying outside the branch neighborhoods $B(p^{(j)},\rho/2)$ and chosen so thin that
\[
\left\|\int_{B_{c}^{\mathrm{tail}}} v(x)\,\dd x\right\|\le \frac{c^2}{3}.
\]
\end{comment}
Now we define the full fixed background domain
\begin{equation}\label{eqndefscr}
S_{c,R_c}
:=
(T_{R_c} \cup J_c) \cup (\widetilde{T}_{c} \cup \widetilde{J}_{c}) \cup  O_c \cup S_c^{\mathrm{bd}}.
\end{equation}
Note this set  $S_{c,R_c}$ is open and connected, as all points are connected to  $O_c.$ Furthermore, the norm of $\int_{S_{c,R_c}} v(x)\,\dd x$ is bounded above by the sum of the norms of the integrals of $v(x)$ over the six domains appearing on the right-hand side of \eqref{eqndefscr}. Consequently, by \eqref{eqnintoc}, \eqref{eqninttrc3}, \eqref{eqninttrc4}, \eqref{eqninttc2}, \eqref{eqninttc3}, and \eqref{eqnscbd},
\begin{equation}\label{eqnintscrcpart}
\left\|\int_{S_{c,R_c}} v(x)\,\dd x\right\|
\le 4c^2.
\end{equation}
\begin{comment}
In particular,
\[
\left|\int_{S_{c,R(c)}} v(x)\,\dd x\right|=o(c)
\qquad (c\to 0^+).
\]
\end{comment}
Moreover, for every multi-index $\alpha \in \N^n$ with $|\alpha|\ge m+1$, \eqref{eqninttrc1} gives
\begin{equation}\label{eqnintscrcpartm1}
\int_{S_{c,R_c}} e^{2\langle \alpha+\mathbf{1}, x \rangle}\,\dd x
\ge
\int_{T_{R_c}} e^{2\langle \alpha+\mathbf{1}, x \rangle}\,\dd x
=
+\infty.
\end{equation}

\smallskip

\noindent
\textbf{Step 2. Adding the variable domains $V_{j,c}(t)$.}

\smallskip

Recall $a_1,\cdots,a_M$ denote the positive coefficients given in Corollary~\ref{cor:discrete-A}. Set 
$$a^*:=2 \max_{1 \leq j \leq M} a_j>0.$$
Note $a^*$ and $\rho$  depend only on $m$ and $n$.  Then we can shrink $c_0$ to a  smaller number, which only depends on $m$ and $n$, such that

$$V_{j, c}^*:=p_j+\left((-\delta_c,\ a^* c/\sigma_c)\times Q_c\right)  \subseteq B(p_j,\rho)~~~~\forall  1 \leq j \leq M, 0< c < c_0.$$
For $0<c<c_0$, we define for each $1\leq j\leq M$,
\[
V_{j,c}(t):=p_j+\bigl((-\delta_c,\ t/\sigma_c)\times Q_c\bigr)
\subseteq V_{j,c}^*,\qquad t\in(0,a^*c).
\]
Let $U_{j,c}$, $1\leq j\leq M$, be the anchor domains constructed in Step 1. Note that
\begin{equation}\label{eqnajcvjc}
	V_{j,c}(t)\cap U_{j,c}
	=p_j+\bigl((-\delta_c,0)\times Q_c\bigr),
\end{equation}
which is independent of $t$. Consider the difference
\begin{equation}\label{eqndefcjc}
	W_{j,c}(t):=V_{j,c}(t)\setminus U_{j,c}
	=
	p_j+\bigl([0,\ t/\sigma_c)\times Q_c\bigr).
\end{equation}
Its volume is
\begin{equation}\label{eqnvolcjc}
	\mathrm{Vol}_{n}(W_{j,c}(t))=t,~\text{as}~\mathrm{Vol}_{n-1}(Q_c)=\sigma_c.
\end{equation}
%Because $\delta_c\to 0$, $\varepsilon_c\to 0$, and $t/\sigma_c=O(c/\sigma_c)\to 0$ when $t\sim c$, for sufficiently small
%$c$ and all $t$ in a neighborhood of $ca_j$ we have
%\[
Since $V_{j,c}(t)\subseteq B(p_j,\rho),$ and  $\{B(p_j,2\rho)\}_{j=1}^M$ are
pairwise disjoint, we see
\begin{equation}\label{eqnvjvk}
W_{j,c}(t) \cap W_{k,c}(t') =\emptyset,~~~\forall 1 \leq j \neq k \leq M~\text{and}~t, t'\in(0,a^*c).
\end{equation}
Moreover, by construction, $V_{j,c}(t)$ is disjoint from
$S_{c,R_c}\setminus U_{j,c}$ for each $j$. Indeed, \eqref{eqngammabpj}
implies that $V_{j,c}(t)\cap \Gamma_{j,c}=\emptyset$. Furthermore, every
other subdomain of $S_{c,R_c}$ lies outside $B(p_j,\rho)$. As a result,
for each $1\le j\le M$,
\begin{equation}\label{eqnvjajcj}
	V_{j,c}(t)\cap S_{c,R_c}=V_{j,c}(t)\cap U_{j,c},\quad
	V_{j,c}(t)\cup S_{c,R_c}=S_{c,R_c}\cup W_{j,c}(t),\quad
	S_{c,R_c}\cap W_{j,c}(t)=\emptyset.
\end{equation}

%he extra exit corridor
%$B_{c,\mathrm{exit}}$, from the bridge $B_{c}^{\mathrm{tail}}$, and from the tail itself.

\subsection{Analysis on their union $D_{c}(T)$}\label{sec33}

%\bigskip

We keep the same notation as above, and further set
$$\mathcal{T}_c=\{T=(t_1, \cdots, t_M) \in \R^M: 0< t_j< a^*c~\text{for all}~1 \leq j \leq M\}.$$
For each $T \in \mathcal{T}_c$, we define the open set
\begin{equation}\label{eqndefdc}
D_{c}(T):=S_{c,R_c}\cup \left(\bigcup_{j=1}^M V_{j,c}(t_j)\right).
\end{equation}
Since each branch $V_{j,c}(t_j)$ intersects the fixed background domain $S_{c,R_c}$, it follows that $D_c(T)$ is connected.
We also record the following fact in the next remark. Readers may note that it is not needed for most of the subsequent arguments, and will be used only in the proof of part~(2) of Theorem~\ref{thm2}.

\begin{remark}\label{rmkminsxdc}
	Let $s_0$ be as defined in \eqref{eqns0}, and let $\mu$ be as introduced in \eqref{eqndeftildetc}. Then the following property holds for $D_c$:
	\begin{equation}\label{eqninfsxdc}
		\inf_{x\in D_c} S(x)=\mu cs_0.
	\end{equation}
To see this, recall that $D_c$ consists of $T_{R_c}\cup J_c$, $\widetilde{T}_c\cup\widetilde{J}_c$, and some additional subdomains contained in $\mathcal{B}_2$. By Remark~\ref{rmklcjc}, if $x\in T_{R_c}\cup J_c$, then $x\in(0,\infty)^n$ and hence $S(x)>n>\mu cs_0$. By the definition of $s_0$ in \eqref{eqns0}, we have $S(x)>s_0>\mu cs_0$ for all $x\in\mathcal{B}_2$. Then, \eqref{eqninfsxdc} follows from Remark~\ref{rmktildelj}.
\end{remark}

Next, by \eqref{eqnvjajcj} and \eqref{eqnvjvk}, we note that for every integrable $f$,
\begin{equation}\label{eqnintdc}
	\int_{D_c(T)} f
	=
	\int_{S_{c,R_c}} f
	+
	\sum_{j=1}^M \int_{W_{j,c}(t_j)} f.
\end{equation}
Now we define a moment map as follows. For each $T\in\mathcal{T}_c$, set
\begin{equation}\label{eqnphi}
	\Phi_c(T):=
	\left(
	\int_{D_c(T)} e^{2\inner{\alpha+\mathbf{1}}{y}}\,\dd y
	\right)_{\alpha\in\I}
	\in\R^N.
\end{equation}
It is clear that $\Phi_c$ is a well-defined map on $\mathcal{T}_c$.
Set
\[
r_c:=\frac12 c\min_{1\le j\le M} a_j>0,
\]
and denote by $\mathcal{B}_N(r_c)$ the ball in $\R^N$ centered at $0$ with radius $r_c$:
\[
\mathcal{B}_N(r_c):=
\left\{
\tau=(\tau_1,\cdots,\tau_N)\in\R^N:
\sum_{j=1}^N \tau_j^2<(r_c)^2
\right\}.
\]
Write
\[
T_c^0:=(ca_1,\cdots,ca_M), \qquad
\tau^\sharp:=(\tau,0,\cdots,0)\in\R^M
\quad \text{for each } \tau\in\R^N,
\]
where the last $M-N$ entries are zero.
\begin{comment}
Let
\[
E:=\operatorname{span}\{e_1,\cdots,e_N\}\subset \R^M.
\]
We vary only the first $N$ coordinates.

Set
\[
r_c:=\frac12 c\min_{1\le j\le N} a_j.
\]
\end{comment}
By the definitions of $\mathcal{T}_c$ and $r_c$,
\[
T_c^0+\tau^\sharp\in\mathcal{T}_c,
\qquad \forall\,\tau\in\Ol{\mathcal{B}_N(r_c)}.
\]
As a result, we obtain a well-defined map $\Psi_c$ from
$\Ol{\mathcal{B}_N(r_c)}$ to $\R^N$ given by
\begin{equation}\label{eqnpsi}
	\Psi_c(\tau):=\Phi_c(T_c^0+\tau^\sharp),
	\qquad \tau\in\Ol{\mathcal{B}_N(r_c)}.
\end{equation}
\begin{comment}
for $e\in E$ with $|e|<r_c$, each of the first $N$ coordinates
\[
t_j=ca_j+e_j
\]
remains positive and still satisfies $t_j\sim c$. The remaining coordinates are fixed:
\[
t_j=ca_j,\qquad j>N.
\]
%\textbf{Step 9. Moment map.}
Restrict to the affine slice $t^{0,c}+E$:
\[
G_{c,R(c)}(e):=F_{c,R(c)}(t^{0,c}+e),
\qquad e\in E,\ |e|<r_c.
\]
\end{comment}
Our goal is to solve
\[
\Psi_c(\tau)=cA~~~~\text{with}~~~ \tau \in \Ol{\mathcal{B}_N(r_c)}.
\]
The strategy is to make use the Brouwer fixed point type theorem--Proposition \ref{prop:brouwer}. Before applying it, we need to verify that both assumptions of Proposition \ref{prop:brouwer} are satisfied. In order to do that, we define a function $I_{j,\alpha,c}$ (which we will write as $I_{j,\alpha}$ for simplicity, although it depends on $c$)
%\textbf{Step 10. Jacobian estimate.}
for each $\alpha\in \I$ and $1 \le j\le N$, 
\[
I_{j,\alpha}(\tau):=\int_{W_{j,c}(ca_j+\tau_j)} e^{2\inner{\alpha+\mathbf{1}}{y}}\,\dd y,~~~~~~\text{with}~\tau=(\tau_1, \cdots, \tau_N) \in \Ol{\mathcal{B}_N(r_c)}.
\]
Then by \eqref{eqnintdc}, \eqref{eqnphi} and \eqref{eqnpsi}, we have $\Psi_c=\left(
\Psi_{c, \alpha}
\right)_{\alpha\in \I}$, where
\begin{equation}\label{eqnpsica}
\Psi_{c, \alpha}(\tau)=\int_{S_{c,R_c}} e^{2\inner{\alpha+\mathbf{1}}{y}}\,\dd y+\sum_{j=1}^N I_{j,\alpha}(\tau)
+
\sum_{j=N+1}^M \int_{W_{j,c}(ca_j)} e^{2\inner{\alpha+\mathbf{1}}{y}}\,\dd y.
\end{equation}
Writing $y=p_j+(x_1,x')$, we have
\[
I_{j,\alpha}(\tau)
=
\int_0^{\frac{ca_j+\tau_j}{\sigma_c}}\int_{U_c}
e^{2\inner{\alpha+\mathbf{1}}{\,p_j+(x_1,x')}}\,\dd x'\,\dd x_1.
\]
Then it is clear that $I_{j,\alpha}$ is a  $C^0-$function in $\Ol{\mathcal{B}_N(r_c)}$ and is $C^1$ in $\mathcal{B}_N(r_c)$. Moreover, writing $\tau_{j, c}:=\frac{ca_j+\tau_j}{\sigma_c},$ we have in $\mathcal{B}_N(r_c):$
\begin{equation}\label{eqnijkderivative}
\frac{\partial I_{j,\alpha}}{\partial \tau_k} \equiv 0,~\forall 1 \leq j \neq k \leq N.
\end{equation}

\begin{equation}\label{eqnijjderivative}
\frac{\partial I_{j,\alpha}}{\partial \tau_j}(\tau)
=
\frac1{\sigma_c}\int_{U_c}
e^{2\inner{\alpha+\mathbf{1}}{\,p_j+(\tau_{j, c},x')}}\,\dd x', ~\forall 1 \leq j \leq N.
\end{equation}
To analyze the integrand in \eqref{eqnijjderivative}, we first establish a property of the function
$v_{\alpha}(y):=e^{2\inner{\alpha+\mathbf{1}}{y}}.$

\begin{lemma}\label{lmheal}
Let $\kappa_2$ be as in \eqref{eqnka12}.  For $y, \hat{y} \in  \mathcal{B}_2$, we have
\begin{equation}\label{eqnmvt}
|v_{\alpha}(y)-v_{\alpha}(\hat{y})|\leq  \kappa_2 \|y-\hat{y}\|.
\end{equation}
Consequently, taking $y=p_j$ and $\hat{y}=p_j+(\tau_{j, c},x')$ with $x' \in Q_c,$ we get
\begin{equation}\label{eqnepjdiff}
\left| e^{2\inner{\alpha+\mathbf{1}}{\,p_j}}- e^{2\inner{\alpha+\mathbf{1}}{\,p_j+(\tau_{j, c},x')}}\right| \leq   \hat{\lambda}_c,
\end{equation}
where $ \hat{\lambda}_c:=\kappa_2 (a^* \frac{c}{\sigma_c}+\sqrt{n-1}\varepsilon_c).$
\end{lemma}

\begin{proof}
Since $\mathcal{B}_2$ is convex, we can apply the mean value theorem to $v_{\alpha}(ty+(1-t)\hat{y}), t\in [0, 1]$ to obtain \eqref{eqnmvt}.  
To prove \eqref{eqnepjdiff}, we note that since $|\tau_j| \leq r_c,$ we have $0< \tau_{j, c}< \frac{a^* c}{\sigma_c},$ and thus
$p_j+(\tau_{j, c},x') \in V_{j, c}^* \subseteq B(p_j,\rho) \subseteq \mathcal{B}_2$ for every $x' \in Q_c.$
Therefore, in \eqref{eqnmvt}, we may take $y=p_j$ and $\hat{y}=p_j+(\tau_{j,c},x')$, where $x'\in Q_c$, to obtain

\begin{equation}\label{eqnmvttoepj}
|v_{\alpha}(y)-v_{\alpha}(\hat{y})|\leq \kappa_2 \|y-\hat{y}\| =\kappa_2 \left(|\tau_{j, c}|^2+ \|x'\|^2\right)^{1/2} \leq \kappa_2 \left(|\tau_{j, c}|+\|x'\|\right).
\end{equation}
Finally, we note that $|\tau_{j,c}|<\dfrac{a^*c}{\sigma_c}$ and $\|x'\|\le \sqrt{n-1}\,\varepsilon_c$. These facts, together with \eqref{eqnmvttoepj}, yield \eqref{eqnepjdiff}.
\end{proof}

\begin{comment}
Since $|x'|\le \varepsilon_c$ and $t/\sigma_c=O(c/\sigma_c)\to 0$ uniformly on the parameter ball,
\[
I'_{j,\alpha}(t)
=
e^{2\inner{\alpha+1}{p^{(j)}}}+O\!\left(\varepsilon_c+\frac{c}{\sigma_c}\right),
\]
uniformly for $|e|<r_c$.
\end{comment}

By \eqref{eqnijjderivative}, since $\mathrm{Vol_{n-1}}(Q_c)=\sigma_c$, we have for every $1 \le j \le N,$

$$\frac{\partial I_{j,\alpha}}{\partial \tau_j} (\tau) - v_{\alpha}(p_j)= \frac{\partial I_{j,\alpha}}{\partial \tau_j} (\tau) -  e^{2\inner{\alpha+\mathbf{1}}{\,p_j}}
=\frac1{\sigma_c}\int_{Q_c}
\left(e^{2\inner{\alpha+\mathbf{1}}{\,p_j+(\tau_{j, c},x')}}- e^{2\inner{\alpha+\mathbf{1}}{\,p_j}}\right)\,\dd x'.$$
By the above equation and Lemma \ref{lmheal}, we get 
\begin{equation}\label{eqnijajdiffestimate}
\left|\frac{\partial I_{j,\alpha}}{\partial \tau_j} (\tau) -  e^{2\inner{\alpha+\mathbf{1}}{\,p_j}} \right| \leq \hat{\lambda}_c~~~\text{for all}~~~ \tau \in \mathcal{B}_N(r_c).
\end{equation}
Write $D \Psi_c$ for the Jacobian matrix of $\Psi_c$ with respect to $(\tau_1, \cdots, \tau_N).$ That is, for each $1 \le j \le N,$ the $j$th column of $D \Psi_c$ is given by $\left(\frac{\partial \Psi_{c, \alpha}}{\partial \tau_j}\right)_{\alpha\in \I}$. Then by \eqref{eqnpsica}, \eqref{eqnijkderivative} and \eqref{eqnijajdiffestimate}, we have
\begin{equation}\label{eqndpsic}
D \Psi_c(\tau)=L+X_c(\tau),~~~~\tau \in \mathcal{B}_N(r_c),
\end{equation}
where $L$ is given by $\eqref{eqndefl}$ and $X_c=\left((X_c)_{ij}\right)_{1 \le i, j \le N}$ is a matrix-valued function in $\mathcal{B}_N(r_c)$ satisfying every entry
$$|(X_c)_{ij}| \leq  \hat{\lambda}_c~\text{in}~\mathcal{B}_N(r_c),~~~1 \le i, j \le N.$$
Consequently, as an operator from $\R^N$ to $\R^N,$ the operator norm of $X_c(\tau)$ satisfies

\[
\|X_c(\tau)\|_{\mathrm{op}} \leq \lambda_c:=N \hat{\lambda}_c,~~~~\forall~\tau \in \mathcal{B}_N(r_c).
\]
This, together with \eqref{eqndpsic}, yields
\begin{equation}\label{eqncondition1psi}
\sup_{\tau \in \mathcal{B}_N(r_c)}\|D \Psi_c(\tau)-L\|_{\mathrm{op}} \le \lambda_c.
\end{equation}
Note that $L$ depends only on $m$ and $n$, and that $\lambda_c=o(c)$ as $c \to 0^+$. By further shrinking $c_0$ to a smaller number, we may assume
\begin{equation}\label{eqncondition2psi}
	\|L^{-1}\|_{\mathrm{op}} \lambda_c\le \frac12,~~~~\forall 0 <c < c_0.
\end{equation}
Note that this new constant $c_0$ can still be chosen to depend only on $m$ and $n$, in view of the expressions for $\lambda_c$ and $\hat{\lambda}_c$.
As a consequence of \eqref{eqncondition1psi} and \eqref{eqncondition2psi},
\begin{equation}\label{eqncondition12psi}
	\sup_{\tau \in \mathcal{B}_N(r_c)}\|D \Psi_c(\tau)-L\|_{\mathrm{op}}\le \frac{1}{2\|L^{-1}\|_{\mathrm{op}}},~~~~\forall 0 <c < c_0.
\end{equation}
%\textbf{Step 11. Base point defect.}
We finally estimate $\Psi_c(\tau)$ at $\tau=0.$ For that, we note
\begin{equation}\label{eqnpsidecomposition}
\Psi_c(0)
=
\int_{S_{c,R_c}} v(y)\,\dd y
+
\sum_{j=1}^M \int_{W_{j,c}(ca_j)} v(y)\,\dd y.
\end{equation}
To estimate the last integrals in \eqref{eqnpsidecomposition}, we note from \eqref{eqnvolcjc} that the volume of $W_{j,c}(ca_j)$ is $ca_j$. Consequently,
\begin{equation}\label{eqnintatzero}
\left\|\int_{W_{j,c}(ca_j)} v(y)\,\dd y-ca_j\,v(p_j)\right\|\leq \Delta_{j, c}:=\int_{W_{j,c}(ca_j)} \|v(y)-v(p_j)\|\,\dd y
%+E_{j,c},
\end{equation}
By \eqref{eqnmvt} in Lemma \ref{lmheal}, we have for $y \in W_{j,c}(ca_j),$
$$|v_{\alpha}(y)-v_{\alpha}(p_j)|\leq  \kappa_2 \|y-p_j\| \leq \hat{\lambda}_c.$$
Since $v(y)$ is an $N-$vector whose components are $v_{\alpha}(y)$, we get %$\|v(y)-v(p_j)\|\le \lambda_c.$ Consequently, where
\[
\|v(y)-v(p_j)\|\le \lambda_c,~\text{and thus}~
\Delta_{j,c}
\le \,ca_j \lambda_c.
\]
We next recall $A$ can be expressed as in \eqref{eqnaincone}.
%$$A=\sum_{j=1}^M a_j\,v(p_j).$$
Then we sum up the inequalities in \eqref{eqnintatzero}  over $1 \le j \leq M,$ and use the triangle inequality in $\R^N$ to obtain
\[
\left\| \sum_{j=1}^M \int_{W_{j,c}(ca_j)} v(y)\,\dd y-cA \right\|
\leq \sum_{j=1}^M \Delta_{j,c} \le  c \lambda_c M \sum_{j=1}^M a_j.
\]
%Also
%\[
%\left\|\int_{S_{c,R(c)}} v(x)\,\dd x\right\|=o(c).
%\]
Combining this with \eqref{eqnintscrcpart} and \eqref{eqnpsidecomposition}, we get
\[
\left\|\Psi_c(0)-cA\right\| \leq \mu_c:= 4c^2+c \lambda_c M \sum_{j=1}^M a_j.
\]
Since $\lambda_c=o(c)$ as $c \to 0^+$, the same is true for $\frac{\mu_c}{c}$. On the other hand, $\lim_{c \to 0^+}\frac{r_c}{c}=\frac{1}{2} \min_{1\le j\le M} a_j>0$ (which depends only on $m$ and $n$). By further shrinking $c_0$ to a smaller number, we may assume
$$\|L^{-1}\|_{\mathrm{op}}~\mu_c \leq \frac{r_c}{2},~~~~\forall 0<c< c_0.$$
Note that this new constant $c_0$ can still be chosen to depend only on $m$ and $n$, in view of the expressions for $r_c$ and $\mu_c$.
As a result of the last two inequalities, for every $0<c< c_0,$
\begin{equation}\label{eqncondition3psi}
\left\|L^{-1}\left(\Psi_c(0)-cA\right)\right\| \le \frac{r_c}{2}.
\end{equation}
From \eqref{eqncondition12psi} and \eqref{eqncondition3psi}, we see  both assumptions of Proposition \ref{prop:brouwer} are satisfied for the map $G=\Psi_c$ with $r=r_c$ and $A^0=cA,$ the model linear map $L$ being given by \eqref{eqndefl}, provided that $0<c< c_0$. Then by the conclusion of  Proposition \ref{prop:brouwer}, for each $0< c<c_0$,
%the model linear map $L$,the target value $cA$, and the radius $r_c$. Hence
there exists some $\tau_c \in \Ol{\mathcal{B}_N(r_c)},$
%\[
%e_{c,R(c)}^*\in E,\qquad |e_{c,R(c)}^*|<r_c,
%\]
such that
\[
\Psi_c(\tau_c)=cA. 
\]
Equivalently, $\Phi_c(T^0_c+\tau_c^\sharp)=cA.$ That is, for  $0<c< c_0,$
set the domain
\[
D^*_{c}:=D_{c}(T^0_c+\tau_c^\sharp),
\]
where $D_{c}(T)$ is the domain as defined in \eqref{eqndefdc}. Then
%we have $D^*_{c}$ is connected and open, and
\begin{equation}\label{eqnintdstar1}
\int_{D^*_{c}} e^{2\inner{\alpha+\mathbf{1}}{y}}\,\dd y=cA_\alpha,
\qquad \text{for every }\alpha\in \I.
\end{equation}
Furthermore, %for every $\alpha$ with $|\alpha|\ge m+1$, the integral over $D^*_{c}$ is infinite 
by the construction of $D_c(T)$ and $D^*_{c}$, and \eqref{eqnintscrcpartm1},

\begin{equation}\label{eqnintdstar2}
	\int_{D^*_{c}} e^{2\inner{\alpha+\mathbf{1}}{y}}\,\dd y=+\infty
	\qquad \text{for every }\alpha \in \N^n~\text{with}~|\alpha|\ge m+1.
\end{equation}

\begin{remark}
In the above argument, although we have shrunk the constant $c_0$ at several steps, at each step we can shrink $c_0$ to a smaller number that depends only on $m$ and $n$. Therefore, we may write the final constant as 
$$c_0:=c_{m,n},$$ 
which depends only on $m$ and $n$. We briefly summarize what we have proved so far: there exists a constant $c_{m,n}$ such that, for every $0<c<c_{m,n}$, there exists a domain $D_c^*$ satisfying \eqref{eqnintdstar1} and \eqref{eqnintdstar2}.
\end{remark}

%\end{proof}

\subsection{Lift to a Reinhardt domain and proof of Theorem \ref{thm1}}\label{sec34}

The logarithmic lift $\mathcal{L}(D^*_{c})$ of the set $D^*_{c}$ finally gives the desired  Reinhardt domain. More precisely, we have %As before, for any subset $D \subseteq \R^n, $ we define the logarithmic lift of $D$ be the following set in $\C^n:$
%$$\mathcal{L}(D):= \{z\in \C^n:\ (\log|z_1|,\cdots,\log|z_n|)\in E\}.$$

\begin{theorem}\label{thm:final-domain}
Let $c_{m,n}$ and  $D^*_{c},~0<c< c_{m, n},$ be as above. For  each $0< c< c_{m, n}$, define
\[
\Omega^*_c:=\mathcal{L}(D^*_{c})= \{z=(z_1, \cdots,z_n) \in \C^n:\ (\log|z_1|,\cdots,\log|z_n|)\in D^*_{c}\}.
\]
Then $\Omega^*_c$ is a connected Reinhardt domain such that
\begin{enumerate}
	\item[(1)]The Bergman space $A^2(\Omega^*_c)$ of $\Omega^*_c$ is spanned by $\{z^{\alpha}\}_{\alpha\in \I}.$ In particular, the  Bergman space separates points. %in $\Omega_c$.

\item[(2)]The Bergman kernel of $\Omega^*_c$ is given by 
\[
K_{\Omega^*_c}(z,z)=\frac{1}{c}(1+\|z\|^2)^m.
\]
\end{enumerate}
\end{theorem}

\begin{proof}
Let $0< c< c_{m,n}$. By the construction of $\Omega^*_c$, since $D^*_{c}$ is connected and open, it follows that $\Omega^*_c$ is a connected Reinhardt domain. By a classical expansion theorem (see, for example, \cite[Theorem 1.5, p.~46]{Ra}), every holomorphic function $f$ on $\Omega^*_c$ admits a unique Laurent series expansion which converges uniformly on compact subsets:
\begin{equation}\label{eqnfexpansion1}
	f=\sum_{\alpha \in \mathbb{Z}^n} a_{\alpha} z^{\alpha} \quad \text{on } \Omega^*_c, \quad~\text{with}~~a_{\alpha} \in \C.
\end{equation}
%where the infinite sum on the right hand side in \eqref{eqnfexpansion1} converges uniformly on every compact subset of $\om^*_c.$ Moreover, the coefficients $a_{\alpha} \in \mathbb{C}$ are uniquely determined by $f.$ 
We next prove

\begin{lemma}\label{lmfexpansion}
Let $f \in A^2(\Omega^*_c)$ with Laurent series expansion as in \eqref{eqnfexpansion1}. Then $a_{\alpha}=0$ for all $\alpha \in \Z^n \setminus \N^n.$
\end{lemma}

\begin{proof}
	As indicated in Remark~\ref{rmktildetc}, we make use of the second tail domain $\widetilde{T}_{c}$ constructed earlier in \eqref{eqndeftildetc}.
	Set $\mathcal{P}_{c}:= \mathcal{L}(\widetilde{T}_c) \subseteq \Omega^*_c$ and write
	$$\mathcal{A}_c=\mathcal{A}_c(\mu):=\{z=(z_1, \cdots,z_n) \in \C^n:  \mu c s_0
	< \sum_{i=1}^n|z_i|^2 <c s_0\},\quad \text{and} \quad \mathcal{V}_c:=\{z \in \mathcal{A}_c: \prod_{j=1}^n z_j=0 \}.$$
	By the definitions of $\mathcal{L}$ and $\widetilde{T}_c$, we have
	\begin{equation}\label{eqnacvp}
		\mathcal{P}_{c}=\mathcal{A}_c \setminus \mathcal{V}_c.
	\end{equation}
	Since $\widetilde{T}_c \subseteq D^*_{c}$, we have $\mathcal{P}_{c} \subseteq \Omega^*_c$, and hence $f \in A^2(\mathcal{P}_{c})$. By \eqref{eqnacvp} and the Riemann removable singularity theorem, it follows that $f \in A^2(\mathcal{A}_c)$.
	By Hartogs' extension theorem, $f$ extends holomorphically to the ball $\mathcal B_c:=\{z \in \C^n: \sum_{i=1}^n|z_i|^2 <c s_0\}$, and thus admits a power series expansion on $\mathcal B_c$:
	$$f=\sum_{\alpha \in \mathbb{N}^n} c_{\alpha} z^{\alpha}, \quad c_{\alpha} \in \C.$$
	Restricting this expansion to $\mathcal{P}_c$, we obtain a power series expansion of $f$ on $\mathcal{P}_c$.
	On the other hand, restricting \eqref{eqnfexpansion1} to $\mathcal{P}_c \subseteq \Omega^*_c$ yields a Laurent series expansion of $f$ on $\mathcal{P}_c$.
	By the uniqueness of the Laurent expansion on the Reinhardt domain $\mathcal{P}_c$, we conclude that $a_{\alpha}=0$ for all $\alpha \in \Z^n \setminus \N^n$.
\end{proof}

Next for any multi-index $\alpha \in \N^n$, by Proposition \ref{prop:intloglift},

\[
\int_{\Omega^*_c}|z^\alpha|^2\,\dd V_{2n}(z)
=
(2\pi)^n
\int_{D^*_{c}}  e^{2\langle \alpha+\mathbf{1}, x \rangle}\,\dd x.
\]
%By Theorem~\ref{thm:geometric-realization}, 
Then %for $|\alpha|\le m$, i.e., $\alpha \in \I$ 
by the above equation, \eqref{eqnintdstar1} and the property of Reinhardt domains, we have
\begin{equation}\label{eqnintzalm}
	\int_{\Omega^*_c} z^\alpha \overline{z}^{\beta}\,\dd V_{2n}(z)
	=
	(2\pi)^n \delta_{\alpha\beta}\,(cA_\alpha), \quad \forall\, \alpha,\beta \in \I,
\end{equation}
where $\delta_{\alpha\beta}$ denotes the Kronecker delta.
%this integral equals $(2\pi)^n(cA)_\alpha$, hence is finite.
%For $|\alpha|\ge m+1$, the corresponding logarithmic integral is infinite, because the tail is included in $D_{c,R(c)}$.
Likewise, using \eqref{eqnintdstar2}, we get 
\begin{equation}\label{eqnintzabmp1}
\int_{\Omega^*_c}|z^\alpha|^2\,\dd V_{2n}(z)
= +\infty
\qquad ~\text{if}~|\alpha|\ge m+1.
\end{equation}
Let $f \in A^2(\Omega^*_c)$. By Lemma \ref{lmfexpansion},  $f$ admits a power series expansion which converges uniformly on compact subsets:
\begin{equation}\label{eqnfexpansion2}
f=\sum_{\alpha \in \mathbb{N}^n} a_{\alpha} z^{\alpha} \quad \text{on } \Omega^*_c.
\end{equation}
Let $\{D^*_{c,k}\}_{k=1}^{\infty}$ be an exhausting sequence of relatively compact subdomains of $D^*_c$. For each $k \geq 1$, write $\Omega^*_{c,k}:=\mathcal{L}(D^*_{c,k})$, which is clearly a Reinhardt domain. Moreover, $\{\Omega^*_{c,k}\}_{k=1}^{\infty}$ is an exhausting sequence of relatively compact subdomains of $\Omega^*_c$. By the uniform convergence of \eqref{eqnfexpansion2} on $ \om^*_{c, k}$ and the property of Reinhardt domains, we have
$$\int_{\Omega^*_{c,k}} |f|^2 \,\dd V_{2n}(z) \geq  |a_{\alpha}|^2 \int_{\Omega^*_{c,k}}|z^\alpha|^2\,\dd V_{2n}(z).$$
By \eqref{eqnintzabmp1} and checking the limit of both sides in the above equation as $k \to \infty$, we conclude $a_{\alpha}=0$ for all $\alpha \in \mathbb{N}^n$ with $|\alpha|\ge m+1$. Using this and \eqref{eqnintzalm}, one can easily verify that the following gives an orthonormal basis in $A^2(\Omega^*_c):$
\begin{equation}\label{eqnonbomegac}
\{(2\pi)^{\frac{-n}{2}}(cA_\alpha)^{-\frac{1}{2}} z^\alpha \}_{\alpha \in \I}.
\end{equation} 
Consequently, the  Bergman space separates points as it contains all constant and linear functions. This proves part (1) of Theorem \ref{thm:final-domain}. Furthermore, the Bergman kernel of $\Omega^*_c$ can be computed as

%$$\frac{1}{2^n(m+n)!}\,\alpha!\,(m-|\alpha|)! $$
\begin{equation}\label{eqnbergmankernelom}
\begin{split}
K_{\Omega^*_c}&=\frac{1}{c}\sum_{\alpha \in \I}  \frac{|z^\alpha|^2}{(2\pi)^n A_\alpha} \\
&=\frac{1}{c}\sum_{\alpha \in \I} \frac{m!}{\alpha!\,(m-|\alpha|)!} |z^\alpha|^2 \\
&=\frac{1}{c} (1+\|z\|^2)^m
\end{split}
\end{equation}
This establishes part (2) of Theorem~\ref{thm:final-domain}, and thus completes the proof.
\end{proof}

\begin{comment}
Since $\Omega$ is Reinhardt, holomorphic functions on $\Omega$ admit monomial expansions adapted to the Reinhardt
symmetry, and distinct monomials are orthogonal in $A^2(\Omega)$. The divergence of the norms for all monomials
of degree $\ge m+1$ forces every $L^2$ holomorphic function on $\Omega$ to lie in
\[
\operatorname{span}\{z^\alpha:\ |\alpha|\le m\}=P_m.
\]
Hence
\[
A^2(\Omega)=P_m.
\]

The Bergman kernel of the finite-dimensional Hilbert space $P_m$ is, up to a positive constant, the standard reproducing kernel
\[
(1+z\cdot \overline w)^m.
\]
Hence
\[
K_\Omega(z,w)=C(1+z\cdot \overline w)^m
\]
for some $C>0$.
\end{comment}

Theorem \ref{thm1} is then a consequence of Theorem~\ref{thm:final-domain}.

\smallskip

{\bf Proof of Theorem \ref{thm1}.} Set
$C_{m,n}:=\frac{1}{c_{m,n}},~C:=\frac{1}{c}$ and $\Omega_C:=\Omega^*_c.$
Then parts~(1) and~(2) of Theorem~\ref{thm1} follow from the corresponding parts of Theorem~\ref{thm:final-domain}. \qed

\smallskip

We conclude this section with two remarks, which will be used in the proof of Theorem~\ref{thm2}.

\begin{remark}\label{rmkfamilyomegac}
Fix $0< c< c_{m,n}$. Recall that in Step~3 of \S\ref{sec31}, when we construct $\widetilde{T}_{c}=\widetilde{T}_{c}(\mu)$, the parameter $\mu$ can be any number in $(0,1)$. Therefore, by varying $\mu$, we obtain a family $\{\widetilde{T}_{c}(\mu)\}_{\mu \in (0,1)}$ of possible choices. This gives rise to a family of domains  $\{D^*_{c}(\mu)\}_{\mu \in (0,1)}$ and thus their logarithmic lift $\{\Omega^*_c(\mu)\}_{\mu \in (0,1)}$. By our construction and theorem~\ref{thm:final-domain}, every $\Omega^*_c(\mu), \mu \in (0,1),$ satisfies the conclusion in Theorem \ref{thm:final-domain}. Moreover, we will show they are Bergman inequivalent for different choices of $\mu$.  To prove that, a key feature of $\Omega^*_c(\mu)$ that we will use is the following consequence of Remark~\ref{rmkminsxdc}:
\begin{equation}\label{eqninfnormomega}
\inf_{z \in \Omega^*_c(\mu)} \|z\|^2=\inf_{x \in D^*_{c}(\mu)} S(x)=\mu cs_0,
\end{equation}
where $S(x)$ is as defined in \eqref{eqndefs}. 
%To see the above equation hold, we note by the definition of $s_0$ and remarks ???, we have $S(x) \geq c s_0$ for every $x \in D^*_{c}(\mu) \setminus \widetilde{T}_{c}(\mu).$ 
\end{remark}

\begin{remark}\label{rmkintersectdc}
By Remark \ref{rmk1taildomain}, $T_R \cap T_{R'} \neq \emptyset$ for any $R, R' > 1$. Recall $T_{R_c} \subseteq D^*_{c}(\mu).$ We thus have
$$D^*_{c}(\mu) \cap D^*_{c'}(\mu') \neq \emptyset,~~~~\text{for any}~0< c, c'< c_{m,n}~\text{and}~\mu, \mu' \in (0,1).$$
Consequently, $\Omega^*_c(\mu) \cap \Omega^*_{c'}(\mu') \neq \emptyset$ for any $0< c, c'< c_{m,n}~\text{and}~\mu, \mu' \in (0,1).$
Furthermore, by the definition of $\widetilde{T}_c(\mu)$ in \eqref{eqndeftildetc}, for any $0<\mu'<\mu<1$, the set $\widetilde{T}_c(\mu')\setminus \widetilde{T}_c(\mu)$, and hence also the symmetric difference $\widetilde{T}_c(\mu')\Delta \widetilde{T}_c(\mu)$, both contain the open set
\[
\left\{x\in\mathbb{R}^n:\mu'cs_0<S(x)<\mu cs_0\right\}.
\]
By our construction and Remark~\ref{rmkminsxdc}, we conclude that  $D^*_{c}(\mu') \Delta D^*_{c}(\mu)$, and hence $\Omega^*_c(\mu') \Delta \Omega^*_{c}(\mu)$, have nonempty interior for any $0< c < c_{m,n}$ and $0< \mu'< \mu < 1$.
\end{remark}

\smallskip

\begin{comment}
	
\begin{remark}\label{rmk1taildomain}
	Since for each $j$ we have $t - x_j \le |x_j - t| < e^{-\gamma t}$, it follows that $x_j > t - e^{-\gamma t} > R - 1$. Therefore, $T_R \subseteq (0,\infty)^n$. Moreover, for every $R > 1$,
	\[
	\{(x_1,\cdots,x_n) : x_1 = \cdots = x_n \in (R,\infty)\} \subseteq T_R.
	\]
	Consequently, for any $R, R' > 1$, we have $T_R \cap T_{R'} \neq \emptyset$. These facts will be used to prove part~(2) of Theorem~\ref{thm2}.
\end{remark}

\section{Conclusion}

Combining Proposition~\ref{prop:A-in-interior}, Corollary~\ref{cor:discrete-A},
Theorem~\ref{thm:geometric-realization}, and Theorem~\ref{thm:final-domain}, we obtain:

\begin{theorem}
For every integer $m\ge 0$ and every dimension $n\ge 2$, there exists a connected Reinhardt domain
\[
\Omega\subset \C^n
\]
such that its Bergman kernel is
\[
K_\Omega(z,w)=C(1+z\cdot \overline w)^m
\]
for some constant $C>0$.
\end{theorem}

\begin{remark}
The proof above is written in detailed draft form. The most delicate part is the geometric realization theorem, whose
role is to realize the low-degree moment vector after the tail and the bridge have already been incorporated into the
fixed geometry. This ordering is logically preferable to attaching the tail only at the end, because it prevents the
tail from introducing an uncontrolled change in the low-degree moments after the local correction has already been made.
\end{remark}

\end{comment}

\section{Proof of some remarks and Theorem \ref{thm2} and \ref{thm3}}\label{sec4}

\subsection{Proof of the statement in Remark \ref{rmkngreaterthan1}, \ref{rmkbergmanneg1}, \ref{rmkbergmanneg2} and \ref{rmkbergmanequivalence}}\label{sec41}

We start with the proof of Remark \ref{rmkngreaterthan1}.

\smallskip

{\bf Proof of the assertion in Remark \ref{rmkngreaterthan1}.} 
We prove the first assertion by contradiction. Suppose that there exists a domain $\Omega \subseteq \C$ whose Bergman metric $\omega$ has constant positive holomorphic sectional curvature. Then $\omega$ is locally isometric to $\lambda \omega_{FS}$ on $\mathbb{P}^1$ for some constant $\lambda>0$. By the discussion following Theorem~0, $\lambda$ must be a positive integer $m$. Hence there exist small domains $U \subseteq \C \subseteq \mathbb{P}^1$ and $V \subseteq \Omega$, together with a biholomorphic map $f:U\to V$ such that $f^*\omega=m\omega_{FS}$.
Let $K$ denote the Bergman kernel of $\Omega$. It follows that $\log(K\circ f)-m\log(1+|z|^2)$ is harmonic on $U$. On the other hand, since $A^2(\Omega)$ is nontrivial, Wiegerinck's theorem for planar domains \cite[Section~3]{Wi} implies that $A^2(\Omega)$ is infinite-dimensional. Let $\{\phi_j\}_{j=0}^\infty$ be an orthonormal basis of $A^2(\Omega)$.
Using the harmonicity of $\log(K\circ f)-m\log(1+|z|^2)$ and shrinking $U, V$ if necessary, a standard argument yields
$$\sum_{j=0}^{\infty} |h(\phi_j\circ f)|^2=(1+|z|^2)^m \quad \text{on } U,$$
where $h$ is a nowhere-vanishing holomorphic function on $U$. Applying $\frac{\partial^{2k}}{\partial z^k \partial \bar z^k}$ to both sides for any $k>m$, we see each $h(\phi_j\circ f)$ is a holomorphic polynomial of degree at most $m$. It follows that the family $\{h(\phi_j\circ f)\}_{j=0}^\infty$ is linearly dependent, and therefore so is $\{\phi_j\}_{j=0}^\infty$, a contradiction. This proves the first assertion.

To prove the second assertion, again arguing by contradiction, suppose that there exists a Bergman nondegenerate Riemann surface $M$ whose Bergman metric has constant positive holomorphic sectional curvature. Then by Huang--Li \cite[Theorem~1.1]{HLT}, $M$ is biholomorphic to a domain $D \subseteq \mathbb{P}^1$. Set $D_0:=D\setminus\{\infty\}\subseteq \C$. By the Riemann removable singularity theorem,  the Bergman kernel of $D_0$ is the restriction of that of $D$ to $D_0$. Consequently, the Bergman metric of $D_0$ also has constant positive holomorphic sectional curvature. This contradicts the first assertion. \qed

\smallskip

{\bf Proof of the assertion in Remark \ref{rmkbergmanneg1}.} 
Since $E$ is closed, $M \setminus E$ is open.

	\noindent
	To prove connectedness of $M \setminus E$,
	we suppose $M \setminus E = U_1 \cup U_2$ with $U_1, U_2$ nonempty disjoint open sets. 
	Pick $0 \neq \eta \in \A(M)$ and define
	\[
	\omega =
	\begin{cases}
		\eta & \text{on } U_1,\\
		0 & \text{on } U_2.
	\end{cases}
	\]
	Then $\omega \in \A(M \setminus E)$. By negligibility, $\omega$ extends to $\widetilde{\omega} \in \A(M)$. 
	Thus $\widetilde{\omega} = \eta$ on $U_1$ and $\widetilde{\omega} = 0$ on $U_2$, which contradicts the identity theorem and the connectivity of $M$. Hence $M \setminus E$ is connected.
To prove density of $M \setminus E$, we
	suppose $E$ contains a nonempty open set $V$. For any $\eta \in \A(M)$,
	\[
	\|\eta\|_{M}^2 = \|\eta\|_{M \setminus E}^2,
	\]
	so $\|\eta\|_{E}^2 = 0$, hence $\eta \equiv 0$ on $V$. By the identity theorem, $\eta \equiv 0$ on $M$, contradicting $\A(M) \neq \{0\}$. Thus $E$ has empty interior, and thus $M \setminus E$ is dense. \qed

\smallskip

{\bf Proof of the assertion in Remark \ref{rmkbergmanneg2}.} 
We first give a more precise formulation of Remark \ref{rmkbergmanneg2}, and then prove it. In the special case where $M$ is a bounded domain in $\C^n$, the result was established in \cite{ETX1}, but the argument there does not extend directly to the setting of general complex manifolds. We therefore provide a proof in full generality.

\begin{proposition}\label{prop:rmk15}
Let $M$ be an $n$-dimensional complex manifold with nontrivial Bergman space, and let $M'$ be a domain in $M$. Assume that the Bergman kernel of $M'$ is the restriction of that of $M$ to $M'$. Then the following restriction map is a unitary isomorphism:
	\[
	\Res : A^2_{(n,0)}(M) \longrightarrow A^2_{(n,0)}(M'), \qquad \omega \longmapsto \omega|_{M'}.
	\]
By definition, $M \setminus M'$ is a Bergman negligible subset of $M$.
\end{proposition}

\begin{proof}
Set $\mathcal H:=A^2_{(n,0)}(M)$ and $\mathcal H':=A^2_{(n,0)}(M')$, and let $K_M$ and $K_{M'}$ denote the Bergman kernels. By assumption, we have
\begin{equation}\label{eqnkmequalkmp}
	K_M|_{M'}=K_{M'}.
\end{equation}
For each $w^* \in M'$, fix a local holomorphic coordinate chart $w=(w_1,\cdots,w_n)$ near $w^*$, and write $d\overline w= d\overline w_1 \wedge \cdots \wedge d\overline w_n.$ Then, by the properties of the Bergman kernel, there exist two unique elements
	$\hat{K}_M(\cdot, w^*)\in \mathcal H$ and $\hat{K}_{M'}(\cdot, w^*)\in \mathcal H'$ such that
	\[
	K_M(\cdot, w^*)
	= \hat{K}_M(\cdot, w^*)\, d\overline w\big|_{w^*}, \quad 
	K_{M'}(\cdot, w^*)
	= \hat{K}_{M'}(\cdot, w^*)\, d\overline w\big|_{w^*}.
	\]
By the assumption on the two Bergman kernels, 
\begin{equation}\label{eqnbergmanequal}
\hat{K}_{M'}(\cdot, w^*)=\hat{K}_M(\cdot, w^*)|_{M'},~\text{for~all}~w^*\in M'.
\end{equation}
Set
	\[
	\mathcal D:=\mathrm{Span}\{\hat{K}_M(\cdot, w^*): w^*\in M'\}\subseteq \mathcal H,
	\quad
	\mathcal D':=\mathrm{Span}\{\hat{K}_{M'}(\cdot, w^*): w^*\in M'\}\subseteq \mathcal H'.
	\]
We define a linear operator $\mathcal U:\mathcal D \to \mathcal D'$ as follows. For
\begin{equation}\label{eqnomexpression}
	\phi=\sum_{j=1}^s c_j \hat{K}_M(\cdot, w_j^*) \in \mathcal D,~~~\text{where }  s \in \Z_+,~ w^*_j \in M' \text{ and } c_j \in \C,
\end{equation}
we set
\[
\mathcal U\phi:=\sum_{j=1}^s c_j \hat{K}_{M'}(\cdot, w_j^*) \in \mathcal D'.
\]
It follows from \eqref{eqnbergmanequal} that $\mathcal U$ is well defined. Moreover, we claim that $\mathcal U$ is surjective and isometric from $\mathcal D$ onto $\mathcal D'$. The surjectivity is clear. To verify that $\mathcal U$ is isometric, fix any $\zeta^*, \eta^* \in M'$. Let $\widetilde{K}_M$ denote the representing function of $K_M$ in local coordinates at $\zeta^*$ and $\eta^*$, as in \eqref{eqnbergmanlocal}, and similarly for $\widetilde{K}_{M'}$. By assumption \eqref{eqnkmequalkmp}, we have $\widetilde{K}_M(\zeta^*, \eta^*)=\widetilde{K}_{M'}(\zeta^*, \eta^*)$. Then, by the reproducing property of the Bergman kernel,
\[
\langle \hat{K}_M(\cdot, \eta^*),\hat{K}_M(\cdot, \zeta^*)\rangle_{\mathcal H}
= \widetilde{K}_M(\zeta^*, \eta^*)
= \widetilde{K}_{M'}(\zeta^*, \eta^*)
= \langle \hat{K}_{M'}(\cdot, \eta^*),\hat{K}_{M'}(\cdot, \zeta^*)\rangle_{\mathcal H'}.
\]
This shows that $\mathcal U$ is isometric.
%Here $\widetilde{K}_M$ denotes the representing function of $K_M$ as in \eqref{eqnbergmanlocal}.

We next claim $\mathcal D$ and $\mathcal D'$ are dense. Indeed, if $\phi \in\mathcal H$ is orthogonal to $\mathcal D$, then for all $w^*\in M',$ we have $\langle \phi,\hat{K}_{M}(\cdot, w^*)\rangle_{\mathcal H}=0$, and thus $\phi (w^*)=0.$ Then by the identity theorem,  $\phi=0$.
Similarly, if $\phi'\in\mathcal H'$ is orthogonal to $\mathcal D'$, then $\phi' \equiv 0$ on $M'$. Thus both $\mathcal D$ and $\mathcal D'$ are dense.

By the density of $\mathcal D$ and $\mathcal D'$, the operator $\mathcal U$ extends uniquely to a unitary operator from $\mathcal H$ onto $\mathcal H'$, which we continue to denote by $\mathcal U$.
On the other hand, let $\phi\in\mathcal D$ be expressed as in \eqref{eqnomexpression}. Then, for any $z\in M'$,
\[
(\mathcal U\phi)(z)=\sum_{j=1}^s c_j \hat{K}_{M'}(z, w_j^*)
=\sum_{j=1}^s c_j \hat{K}_M(z,w_j^*)
=\phi(z).
\]
Thus $\mathcal U\phi=\phi|_{M'}=\Res(\phi)$, for all $ \phi \in \mathcal D$. Since
$\|\Res(\phi)\|_{\mathcal H'}\le \|\phi\|_{\mathcal H}$ for every
$\phi\in\mathcal H,$
the restriction map $\Res:\mathcal H\to\mathcal H'$ is continuous. Therefore, by the density of $\mathcal D$ and the continuity of the two operators, we conclude that $\mathcal U=\Res$ on $\mathcal H$. Hence $\Res$ is a unitary isomorphism, and by definition, $M\setminus M'$ is Bergman negligible.
\end{proof}

{\bf Proof of the assertion in Remark \ref{rmkbergmanequivalence}.} Again we first give a more precise formulation of the remark, and then prove it. In the following, $K_{M_1}$ and $K_{M_2}$ denote the Bergman kernel of the two complex manifolds $M_1$ and $M_2$, respectively.

%\bigskip

\begin{proposition}\label{prop:rmk18}
Let $M_1$ and $M_2$ be complex manifolds with nontrivial Bergman spaces. Assume $M_1 \sim_B M_2$. Then there exists a unitary isomorphism from $A^2_{(n,0)}(M_1)$ to $A^2_{(n,0)}(M_2)$. Moreover, there exist open dense subsets $D_1 \subseteq M_1$ and $D_2 \subseteq M_2$, and a biholomorphic map $f : D_1 \to D_2$ such that
\[
K_{M_1}\big|_{D_1} = f^*\bigl(K_{M_2}\big|_{D_2}\bigr).
\]
\end{proposition}

\begin{proof}
By the definition of Bergman equivalence, there exists a finite sequence of complex manifolds
\[
M_1=X_1,\; X_2,\;\cdots,\;X_k=M_2
\]
such that for each $1\le j\le k-1$, either $X_j\leqB X_{j+1}$ or $X_{j+1}\leqB X_j$. By Definitions~\ref{defnbergmannegligible} and \ref{defnleqb}, each such relation induces a unitary isomorphism between the Bergman spaces of $X_j$ and $X_{j+1}$. Composing these isomorphisms along the chain, we obtain a unitary isomorphism between $A^2_{(n,0)}(X_i)$ and $A^2_{(n,0)}(X_j)$ for any $1\le i,j\le k$, and in particular between $A^2_{(n,0)}(M_1)$ and $A^2_{(n,0)}(M_2)$. Since $M_1$ and $M_2$ have nontrivial Bergman spaces by assumption, it follows that every $X_j$ also has nontrivial Bergman space.
	
By Definition \ref{defnleqb}, for each $1\le j\le k-1$, there exists a biholomorphism
	$\phi_j:U_j\to V_j$,
	where $U_j\subseteq X_j$ and $V_j\subseteq X_{j+1}$ are open subsets obtained by removing Bergman negligible subsets (possibly empty sets), and such that
	\[
	K_{X_j}\big|_{U_j}=(\phi_j)^* \bigl(K_{X_{j+1}}\big|_{V_j}\bigr).
	\]
	By Remark \ref{rmkbergmanneg1}, both $U_j$ and $V_j$ are open dense subsets of $X_j$ and $X_{j+1}$, respectively.

\smallskip

\noindent
{\bf Claim.} For every $1 \le j \le k,$ there exist open dense subsets $E_j\subseteq X_1$ and $G_j\subseteq X_j$, together with a biholomorphism
$f_j:E_j\to G_j,$
such that
\[
K_{X_1}\big|_{E_j} = f_j^*\bigl(K_{X_j}\big|_{G_j}\bigr).
\]

\begin{proof}[Proof of the claim]
We argue by induction on $j$. When $j=1$, simply take $E_1=G_1=X_1$ and $f_1=\mathrm{id}_{X_1}$, and the conclusion holds trivially.
Now assume that for some $1\le j\le k-1$, there exist open dense subsets $E_j\subseteq X_1$ and $G_j\subseteq X_j$, together with a biholomorphism $f_j:E_j\to G_j$, such that
\[
K_{X_1}\big|_{E_j}=f_j^*\bigl(K_{X_j}\big|_{G_j}\bigr).
\]
We prove the statement for $j+1$. To do that, we set
	\[
	E_{j+1}:=f_j^{-1}(G_j\cap U_j), \qquad G_{j+1}:=\phi_j(G_j\cap U_j).
	\]
Since $G_j$ and $U_j$ are open dense subsets of $X_j$, their intersection $G_j\cap U_j$ is also open dense in $X_j$. Because $f_j:E_j\to G_j$ is biholomorphic, it follows that $E_{j+1}$ is open dense in $E_j$, hence open dense in $X_1$. Likewise, since $\phi_j:U_j\to V_j$ is biholomorphic and $G_j\cap U_j$ is open dense in $U_j$, the set $G_{j+1}$ is open dense in $V_j$, hence open dense in $X_{j+1}$. Next define
\[
f_{j+1}:=\phi_j\circ f_j\big|_{E_{j+1}}:E_{j+1}\to G_{j+1}.
\]
Then $f_{j+1}$ is biholomorphic. Moreover, we have
	\[
	f_{j+1}^*\bigl(K_{X_{j+1}}\big|_{G_{j+1}}\bigr)
	=
	f_j^*\!\left((\phi_j)^*\bigl(K_{X_{j+1}}\big|_{G_{j+1}}\bigr)\right) =
	f_j^*\bigl(K_{X_j}\big|_{G_j\cap U_j}\bigr)
	=
	K_{X_1}\big|_{E_{j+1}},
	\]
where in the last step we used the induction hypothesis restricted to $E_{j+1}\subseteq E_j$.
This proves the induction step, and hence the claim follows.
\end{proof}
Finally, set
	\[
	D_1:=E_k\subseteq M_1,\qquad D_2:=G_k\subseteq M_2,\qquad f:=f_k:D_1\to D_2.
	\]
Then $D_1$ and $D_2$ are open dense subsets of $M_1$ and $M_2$, respectively, $f$ is biholomorphic, and
$K_{M_1}\big|_{D_1}=f^*\bigl(K_{M_2}\big|_{D_2}\bigr).$
This proves the proposition.
\end{proof}

%%%%%%%%%%%%%%%%%%%%%%%%%%%%%%%%%%%%%%%%%%%%%
%%%%%%%%%%%%%%%%%%%%%%%%%%%%%%%%%%%%%%%%%%%%%
%%%%%%%%%%%%%%%%%%%%%%%%%%%%%%%%%%%%%%%%%%%%%

\subsection{Proof of Theorem \ref{thm2}, \ref{thm20} and \ref{thm3}}

We first give a proof of Theorem \ref{thm2} and \ref{thm3}. Before that, we prove the following proposition. See a recent related result in \cite{Ya26}; see also rigidity results of a similar flavor in complex analysis established in \cite{LMZ08, MZ21}.

\begin{proposition}\label{propfunitary}
Let $D_1,D_2\subseteq \C^n$ be domains, and assume that their Bergman kernels satisfy
\begin{equation}\label{eqnkdjformula}
K_{D_j}(z,z)=C_j\bigl(1+ \|z\|^2 \bigr)^m,\qquad j=1,2,
\end{equation}
for some constants $C_1,C_2>0$ and some integer $m\ge 1$. Let $U_1 \subseteq D_1$ and $U_2 \subseteq D_2$ be their subdomains.
Let $F: U_1 \to U_2$ be a biholomorphic map  which preserves the Bergman kernel forms of $D_1,D_2$ in the sense that
\begin{equation}\label{eqnfpreservebergmanform}
K_{D_1}=
K_{D_2}(F,F)\,|\det F'(z)|^2~~\text{on}~U_1.
\end{equation}
Then $F$ is a unitary map: $F(z)=Bz$ for some $B \in U(n)$ and $C_1=C_2.$
\end{proposition}

\begin{proof}
By the expressions for their Bergman kernels in \eqref{eqnkdjformula}, we have the Bergman metric of $D_j, j=1, 2,$ equals the restriction of $m$ times the Fubini--Study metric to  $D_j \subseteq \C^n \subseteq \mathbb{P}^n.$ By applying $\partial\bar\partial \log$ to both sides of \eqref{eqnfpreservebergmanform}, we see the map $F$ is a holomorphic isometry between two open subsets of the affine chart $\C^n$ endowed with the Fubini--Study metric. It is well known that (for instance, by  Calabi's theorem~\cite{Ca53}) $F$ extends to a projective unitary transformation. Therefore, on the affine chart $\C^n\subseteq \pj^n$, the map $F$ has the form
\begin{equation}\label{eq:mobius-form}
	F(z)=\frac{Az+b}{\langle c,z\rangle + d},~~\text{where}~U:=
	\begin{pmatrix}
		A & b\\
		c^* & d
	\end{pmatrix}
	\in \U(n+1).
\end{equation}
%\begin{equation}\label{eqnunitary}
%U=
%\begin{pmatrix}
	%A & b\\
	%c^* & d
%\end{pmatrix}
%\in \U(n+1),
%\end{equation}
Here $A$ is an $n\times n$ matrix, $b,c\in \C^n$, $d\in \C$, and $\langle c,z\rangle=\sum_{i=1}^n c_i z_i$.
For such a projective unitary transformation, one has the identity
\begin{equation}\label{eq:FS-identity}
	1+\|F(z)\|^2
	=
	\frac{1+\|z\|^2}
	{|\langle c,z\rangle + d|^2}.
\end{equation}
Note a projective unitary transformation preserves the volume form of the Fubini--Study metric, and this volume form is given by $\bigl(1+ \|z\|^2 \bigr)^{-(n+1)}$, up to a constant, in the affine chart $\C^n$. Using these facts and \eqref{eq:FS-identity}, one also has
\begin{equation}\label{eq:jacobian-formula}
	|\det F'(z)|^2=|\langle c,z\rangle + d|^{-2(n+1)}
\end{equation}
%for some unimodular constant $\chi$ with $|\chi|=1$.
Then substituting \eqref{eq:FS-identity}, \eqref{eq:jacobian-formula} and \eqref{eqnkdjformula}  into \eqref{eqnfpreservebergmanform}, we obtain

\[
C_1(1+\|z\|^2)^m
=
C_2(1+ \|z\|^2)^m
|\langle c,z\rangle + d|^{-2(m+n+1)}.
\]
This simplifies into
$$C_1 |\langle c,z\rangle + d|^{2(m+n+1)}=C_2.$$
This implies that $\langle c,z\rangle + d$ is constant. Consequently, $c=0$. Then, by the expression of $U$ in \eqref{eq:mobius-form}, we have $b=0$, $|d|^2=1$, and $A \in U(n)$. %and hence $\frac{1}{d}A \in U(n)$. 
The conclusion then follows from \eqref{eq:mobius-form} by setting $B:=\frac{1}{d}A \in U(n)$. Consequently, by \eqref{eqnfpreservebergmanform}, $C_1=C_2.$
\end{proof}

We are now ready to prove Theorem \ref{thm2}.

{\bf Proof of Theorem \ref{thm2}.} We will prove part (1) by contrapositive. Let $\Omega \in \mathcal{F}(m,n; C)$ and $\Omega' \in \mathcal{F}(m,n; C'),$ be Bergman equivalent. Then by Remark \ref{rmkbergmanequivalence} or Proposition \ref{prop:rmk18}, there exist open dense subsets $D_1 \subseteq \Omega, D_2 \subseteq \Omega'$ and a biholomorphism $f: D_1 \to D_2$ such that the Bergman kernel forms of $\Omega$ and $\Omega'$ are preserved under $f$. It follows immediately from Proposition \ref{propfunitary} that $C=C'.$ This proves the contrapositive of (1).

%\smallskip

To establish~(2), for each $0< c < c_{m,n}$, let $\Omega^*_c(\mu)$, $\mu \in (0,1)$, be as defined in Remark~\ref{rmkfamilyomegac}. As discussed there, the domains $\Omega^*_c(\mu)$, $\mu \in (0,1)$, all satisfy the conclusion of Theorem~\ref{thm:final-domain}. 
Now set
$$C_{m,n}:=\frac{1}{c_{m,n}}, \quad C:=\frac{1}{c}, \quad \text{and} \quad \Omega_C(\mu):=\Omega^*_c(\mu).$$  
Then, for every $C>C_{m,n}$, we have
$$\{\Omega_C(\mu)=\Omega^*_{1/C}(\mu): \mu \in (0,1)\} \subseteq \mathcal{F}(m,n; C).$$
%To establish~(2), fix $0< c < c_{m,n}$ and $\mu \in (0,1)$, and let $\Omega^*_c(\mu)$ be as defined in Remark~\ref{rmkfamilyomegac}. Set $C:=\frac{1}{c}$ and $\Omega_C(\mu):=\Omega^*_c(\mu)$. 
%For each fixed $0< c < c_{m,n}$, it follows from Remark~\ref{rmkfamilyomegac} that the domains $\Omega^*_c(\mu)$, $\mu \in (0,1)$, all satisfy the conclusion of Theorem~\ref{thm:final-domain}. Consequently, for every $C>C_{m,n}:=\frac{1}{c_{m,n}}$, the family of domains
%\[
%\{\Omega_C(\mu): \mu \in (0,1)\} \subseteq \mathcal{F}(m,n; C).
%\]
We next prove domains in this family are mutually Bergman inequivalent. For that, we fix $\mu , \mu' \in (0,1)$ with $\mu \neq \mu'.$  Seeking a contradiction, suppose $\Omega_C(\mu)$ and $\Omega_C(\mu')$ are Bergman equivalent. Then by Remark \ref{rmkbergmanequivalence} or Proposition \ref{prop:rmk18}, there exist open dense subsets $D(\mu) \subseteq \Omega_C(\mu), D(\mu') \subseteq \Omega_C(\mu')$ and a biholomorphism $f: D(\mu) \to D(\mu')$ such that the Bergman kernel forms of $\Omega_C(\mu)$ and $\Omega_C(\mu')$ are preserved under $f$. By \eqref{eqninfnormomega} and the density property of $D(\mu)$ and $D(\mu')$, we have
\begin{equation}\label{eqninfnormd}
\inf_{z \in D(\mu)} \|z\|^2=\mu cs_0,~~\text{and}~~\inf_{z \in D(\mu')} \|z\|^2=\mu' cs_0,~~\text{where}~~c=\frac{1}{C}.
\end{equation}
On the other hand, it follows from Proposition \ref{propfunitary} that $f$ is a unitary map in $\C^n$. In particular, $f: D(\mu) \to D(\mu')$ is a unitary linear transformation, which preserves the  norm of every point. This contradicts \eqref{eqninfnormd}. Therefore, $\{\Omega_C(\mu): \mu \in (0,1) \} \subseteq \mathcal{F}(m,n; C)$ must be a mutually Bergman inequivalent family of domains. Finally, it follows from Remark \ref{rmkintersectdc} that $\Omega_C(\mu)$ and  $\Omega_{C'}(\mu')$ have nonempty intersection for 
any $C, C'> C_{m,n}$ and  $\mu,\mu' \in (0, 1)$. This completes the proof. \qed

\begin{remark}\label{rmkdeltainteriorpt}
Let $\Omega_C(\mu), \mu \in (0,1),$ be as above. Then it follows from the last part of Remark \ref{rmkintersectdc} that $\Omega_C(\mu) \Delta \Omega_C(\mu')$ has  nonempty interior for any $C > C_{m,n}:=\frac{1}{c_{m,n}}$ and $0< \mu'< \mu < 1.$ This justifies the assertion after Remark \ref{rmkintersectsamek}. 
\end{remark}

\smallskip

We finally prove Theorem \ref{thm3}, and then Theorem \ref{thm20}.

\smallskip

{\bf Proof of Theorem \ref{thm3}.} For $n=1$, the conclusion follows from Theorem~0 and Remark~\ref{rmkngreaterthan1}. For $n \geq 2$, it follows from Theorem~0, Theorems~\ref{thm1} and~\ref{thm2}, together with the discussion following Theorem~0. \qed

\smallskip

{\bf Proof of Theorem \ref{thm20}.}
Let $\omega_1$ and $\omega_2$ denote the K\"ahler forms associated with the Bergman metrics of $D_1$ and $D_2$, respectively. By assumption~(b), the Bergman metrics $\omega_1$ and $\omega_2$ on $D_1$ and $D_2$ are both restrictions of $\omega$. We proceed in four steps.
	
\smallskip
	
\noindent
\textbf{Step 1: transport of an orthonormal basis from $D_1$ to $D_2$.}
	
	Choose points $p_1\in D_1$ and $p_2\in D_2$ with $p_1 \neq p_2$. Since $\hat{D}$ is connected, we can choose a path $\gamma:[0,1]\to \hat{D}$ such that the following two conditions hold:
\begin{enumerate}	
\item[(1)]$\gamma(0)=p_1$ and $\gamma(1)=p_2$.
\item[(2)]There exists a simply connected tubular neighborhood $T\subseteq \hat{D}$ of $\gamma$.
\end{enumerate}

In the following proof, we will make use of the infinite-dimensional complex projective space $\mathbb{P}^{\infty}$, the (formal) Fubini--Study metric $\omega_{FS}$ on $\mathbb{P}^{\infty}$, holomorphic isometric embeddings into $(\mathbb{P}^{\infty}, \omega_{FS})$, and the Bergman--Bochner map. For definitions and preliminaries on these notions, see Section~3 of \cite{ETX2} (see also \cite{HLT}). Calabi's extension theorem and Calabi's rigidity theorem in \cite{Ca53} will also play an important role in the proof; see also \cite{ETX2, HL25} for the precise formulations we use here.
	
Since $D_1$ and $ D_2$ are bounded domains, their Bergman spaces separate points. Let $\{\varphi_j\}_{j=0}^\infty$ be an orthonormal basis of $A^2(D_1)$, and let $\{\psi_j\}_{j=0}^\infty$ be an orthonormal basis of $A^2(D_2)$. Define the associated Bergman--Bochner maps:
	\[
	\Phi_1:D_1\to \mathbb{P}^{\infty},\qquad
	\Phi_1(z)=[\varphi_0(z):\varphi_1(z):\cdots],
	\]
	\[
	\Phi_2:D_2\to \mathbb{P}^{\infty},\qquad
	\Phi_2(z)=[\psi_0(z):\psi_1(z):\cdots].
	\]
In particular, they are holomorphic isometric embeddings from $(D_1, \omega_1)$ and $(D_2, \omega_2)$ into $(\mathbb{P}^{\infty}, \omega_{FS})$. We also note by the property of the orthonormal bases,
the image of $\Phi_j, j=1,2,$ over any open subset of $D_j$ is not contained in any closed  proper projective linear subspace of $\mathbb{P}^{\infty}$.
	%Their pull-backs of the Fubini--Study form are the Bergman metrics of $D_1$ and $D_2$.
	
Write $U_1 \subseteq D_1$ for a small ball centered at $p_1$, and let $\Phi_1|_{U_1}$ denote the restriction of  $\Phi_1$ to $U_1$.
By Calabi's extension theorem, the projective map $\Phi_1|_{U_1}$ extends holomorphically and isometrically along every curve in $\widehat{D}$, and in particular along every curve in $T$, starting from $p_1$. Since $T$ is simply connected, this continuation is single-valued in $T$; denote the resulting holomorphic map on $T$ by $\widetilde\Phi_1:T\to \mathbb{P}^{\infty}$. Since the image of $\widetilde{\Phi}_1$ over $U_1$ is not contained in any closed proper projective linear subspace of $\mathbb{P}^{\infty}$, by analyticity the same holds for its image over any open subset of $T$.
Moreover, $\widetilde\Phi_1$ is an isometric immersion from $(T,\omega)$ to $(\mathbb{P}^{\infty},\omega_{FS})$. Recall that, by assumption~(b), $\omega=\omega_2$ on $D_2$, in particular in a neighborhood of $p_2$.  Then by Calabi's rigidity theorem \cite{Ca53}, after replacing the orthonormal basis $\{\psi_j\}_{j=0}^{\infty}$ of $A^2(D_2)$ by another orthonormal basis obtained from it by a unitary transformation, we may assume that
$\widetilde\Phi_1=\Phi_2$ in a small ball $U_2$ centered at $p_2$.  
We next prove:

\smallskip

{\bf Claim 1.} There exists a sequence of holomorphic functions $\{\widetilde{\varphi}_j\}_{j=0}^\infty$ in $T$ such that
$$\mathrm{(i)}~\widetilde{\varphi}_j=\varphi_j~\text{on}~U_1~~\text{for all }j;~~~~\text{and}~\mathrm{(ii)}~\widetilde\Phi_1=[\widetilde{\varphi}_0: \widetilde{\varphi}_1:\cdots]~\text{on}~T.$$
Consequently, $\sum_{j=0}^\infty |\widetilde{\varphi}_j|^2 = K$ on $T$, where the convergence is locally uniform on $T$.

\smallskip

{\bf Proof of Claim 1.} We first observe, by the existence of the map $\widetilde\Phi_1$, the following two statements hold:
\begin{enumerate}
	\item[(1)]For every $p \in T,$ there exists a small ball $V_p$  centered at $p$, and a sequence of holomorphic functions $\{h_{p, j}\}_{j=0}^\infty$, such that 

\begin{equation}\label{eqncannonp}
\widetilde\Phi_1=[h_{p, 0}: h_{p, 1}:\cdots]~\text{on}~V_p~\text{and}~\sum_{j=0}^\infty |h_{p, j}|^2 = K~\text{uniformly on}~V_p.
\end{equation}

We will call $(V_p, \{h_{p, j}\}_{j=0}^\infty)$ a canonical representation of $\widetilde\Phi_1$ at $p$. In particular, $(U_1, \{\varphi_j\}_{j=0}^\infty)$ is a canonical representation of $\widetilde\Phi_1$ at $p_1.$

\item[(2)] If $(V_p, \{h_{p, j}\}_{j=0}^\infty)$ and $(V_q, \{h_{q, j}\}_{j=0}^\infty)$ are canonical representations of $\widetilde\Phi_1$ at $p, q \in T,$ respectively, and $V_p \cap V_q \neq \emptyset,$ then by \eqref{eqncannonp} and Calabi's rigidity theorem,  there exists a unitary isomorphism $U$ of $\ell^2$ such that
$$(h_{p, 0}, h_{p, 1}, \cdots)=(h_{q, 0}, h_{q, 1}, \cdots) U~\text{on}~V_p \cap V_q.$$
\end{enumerate}
Then, by (1) and (2) and a standard monodromy argument, $F_1 := (\varphi_0, \varphi_1, \cdots)$, and hence in particular every $\varphi_j$, extends holomorphically to $T$.
Write the resulting holomorphic map on $T$ as
\[
\widetilde F_1 = (\widetilde{\varphi}_0, \widetilde{\varphi}_1, \cdots).
\]
In particular, each $\widetilde{\varphi}_j$ is the analytic continuation of $\varphi_j$ from $U_1$ to $T$. Moreover, by \eqref{eqncannonp} we have
\begin{equation}\label{eqntildevarphit}
\sum_{j=0}^\infty |\widetilde{\varphi}_j|^2 = K~\text{on}~T,
\end{equation}
where the series above converges locally uniformly on $T$. This proves Claim 1. \qed

\smallskip

Let  $\widetilde{\Phi}_1$ be as above, and $\{\widetilde{\varphi}_j\}_{j=0}^\infty$ be as in Claim 1.
Since $\widetilde{\Phi}_1=\Phi_2$ on $U_2$, there exists a nowhere-vanishing holomorphic function $h$ on $U_2$ such that
$\psi_j = h\,\widetilde{\varphi}_j$ on $U_2$ for all $j \geq 0.$
Taking squared norms and summing, we obtain 
\[
K_2
=\sum_{j=0}^\infty |\psi_j|^2
=|h|^2\sum_{j=0}^\infty |\widetilde{\varphi}_j|^2
=|h|^2 K~~\text{on}~U_2.
\]
Since $K_2=\lambda K$ on $D_2$, it follows that $|h|^2=\lambda$ on $U_2$.
Because $h$ is holomorphic on the connected set $U_2$, it must be constant. After multiplying all $\psi_j$ by the same unimodular constant, we may assume that $h=\sqrt{\lambda}$. Consequently,
\begin{equation}\label{eqnpsivarphi}
	\psi_j=\sqrt{\lambda}\,\widetilde{\varphi}_j~~\text{on}~U_2,
	\qquad j \geq 0.
\end{equation}

	\smallskip
	
\noindent
\textbf{Step 2: the induced unitary map and the complex moments.}

\smallskip
	
Define a linear map $\mathcal{U}:A^2(D_1)\to A^2(D_2)$
by prescribing $\mathcal{U}(\varphi_j)=\psi_j$ for all $j \geq 0.$ More precisely, if
\begin{equation}\label{eqnfd1}
f=\sum_{j=0}^\infty c_j\varphi_j \in A^2(D_1)~\text{where}~(c_0, c_1, \cdots) \in \ell^2,
\end{equation}
then we define
\begin{equation}\label{eqnufden}
\mathcal{U}(f)=\sum_{j=0}^\infty c_j\psi_j.
\end{equation}
Since $\{\varphi_j\}$ and $\{\psi_j\}$ are orthonormal bases, $\mathcal{U}$ is unitary. 
Let $U_1 \subseteq D_1$ be a small ball centered at $p_1$ as above, and let $\{\widetilde{\varphi}_j\}_{j=0}^\infty$ be as in Claim~1. We prove the following claim about $\mathcal{U}$:

{\bf Claim 2.}  Let $f \in A^2(D_1)$. Then $f|_{U_1}$ admits an analytic extension to $T$, denoted by $\widetilde f$. Moreover,
if $f$ is given by \eqref{eqnfd1}, then
\begin{equation}\label{eqnwtildef}
\widetilde f=\sum_{j=0}^\infty c_j\,\widetilde{\varphi}_j
\qquad\text{on }T,
\end{equation}
As a consequence,
\begin{equation}\label{eqnuflambda}
\mathcal{U}(f)= \sqrt{\lambda} \,\widetilde f~~\text{on}~U_2.
\end{equation}

\smallskip

{\bf Proof of Claim 2.} Consider the series on the right-hand side of \eqref{eqnwtildef}: $g:=\sum_{j=0}^\infty c_j\,\widetilde{\varphi}_j.$ By \eqref{eqntildevarphit} and the fact that $(c_0, c_1, \cdots) \in \ell^2$, it converges pointwise on $T$. Moreover, by the Cauchy--Schwarz inequality,
\[
\sum_{j=m_1}^{m_2} |c_j\widetilde{\varphi}_j|
\le
\Big(\sum_{j=m_1}^{m_2}|c_j|^2\Big)^{1/2}
\Big(\sum_{j=m_1}^{m_2} |\widetilde{\varphi}_j|^2\Big)^{1/2}
\]
Consequently, the series $g$ converges locally uniformly on $T$, and thus defines a holomorphic function on $T$. On $U_1,$
$$g=\sum_{j=0}^\infty c_j\varphi_j=f.$$
So by uniqueness of analytic continuation, it is precisely the analytic extension of $f$ to $T$. This proves \eqref{eqnwtildef}. Finally, we combine \eqref{eqnpsivarphi}, \eqref{eqnufden} and
\eqref{eqnwtildef} to conclude \eqref{eqnuflambda}. This finishes the proof of Claim 2. \qed

\smallskip

Let $\alpha \in \N^n$. As $D_1$ and $D_2$ are bounded, the monomial $z^\alpha$ belongs to both $A^2(D_1)$ and $A^2(D_2)$. Besides, its analytic extension to $T$ is again the same monomial. Then by \eqref{eqnuflambda} in Claim 2,
\[
\mathcal{U}(z^\alpha)=\sqrt{\lambda} z^\alpha.
\]
Therefore, since $\mathcal{U}$ is unitary, we have  for all multi-indices $\alpha,\beta \in \N^n$,
\begin{equation}\label{eqn2integral}
	\int_{D_1} z^\alpha \overline{z^\beta}\,\dd V_{2n}(z)
	=
	\langle z^\alpha,z^\beta\rangle_{A^2(D_1)}
	=
	\langle \mathcal{U}(z^\alpha), \mathcal{U}(z^\beta)\rangle_{A^2(D_2)}
	=
	\lambda\int_{D_2} z^\alpha \overline{z^\beta}\,\dd V_{2n}(z).
\end{equation}

	\medskip
	
\noindent
\textbf{Step 3: uniqueness of the complex moment problem and the conclusion.}

\smallskip
	
Write $\chi_{F}$ for the characteristic function of a set $F \subseteq \C^n$.
Define two Borel measures as follows. For any measurable $E$,
$$\eta_1(E):= \int_{E} d \eta_1, \quad \text{where}~d \eta_1:=  \chi_{D_1} \dd V_{2n};$$
$$\eta_2(E):= \int_{E} d \eta_2, \quad \text{where}~d \eta_2:= \lambda\, \left(\chi_{D_2} \dd V_{2n} \right);$$
Then \eqref{eqn2integral} can be rewritten as
\[
\int_{\C^n} z^\alpha \overline{z^\beta}\,d\eta_1
= \int_{\C^n} z^\alpha \overline{z^\beta}\,d\eta_2
\qquad\text{for all multi-indices}~\alpha,\beta.
\]
By the uniqueness of compactly supported solutions to complex moment problem (see, for example, \cite[Proposition 12.17]{Sc17}), we conclude the two measures coincide: $\eta_1=\eta_2.$
Note by the definition of $\eta_2,$ we have $\eta_2(D_1\setminus D_2)=0,$ and thus $\eta_1(D_1\setminus D_2)=0.$ This yields $D_1\setminus D_2$ has Lebesgue measure zero. Similarly, we can prove
$D_2\setminus D_1$ has Lebesgue measure zero as well. Consequently,  $D_1\cap D_2$ is dense in both $D_1$ and $D_2$, and thus the union $D:=D_1\cup D_2$ is connected. Finally, since $\eta_1(D_1\cap D_2)=\eta_2(D_1\cap D_2)$, we immediately conclude $\lambda=1.$

%%%%%%%%%%%%%%%%%%%%%%%%%%%%%%%%%%%%%%%%%%%%%%%%%%%%%%%%%%%%%%%%%%%
%%%%%%%%%%%%%%%%%%%%%%%%%%%%%%%%%%%%%%%%%%%%%%%%%%%%%%%%%%%%%%%%%%%
%%%%%%%%%%%%%%%%%%%%%%%%%%%%%%%%%%%%%%%%%%%%%%%%%%%%%%%%%%%%%%%%%%%
%%%%%%%%%%%%%%%%%%%%%%%%%%%%%%%%%%%%%%%%%%%%%%%%%%%%%%%%%%%%%%%%%%%
%%%%%%%%%%%%%%%%%%%%%%%%%%%%%%%%%%%%%%%%%%%%%%%%%%%%%%%%%%%%%%%%%%%
%%%%%%%%%%%%%%%%%%%%%%%%%%%%%%%%%%%%%%%%%%%%%%%%%%%%%%%%%%%%%%%%%%%

\medskip

\noindent
\textbf{Step 4: Bergman negligibility of $D\setminus D_1$ and $D\setminus D_2$.}

\smallskip

Now we additionally assume  that $K$ admits a sesqui-holomorphic extension $\widetilde K(\cdot,\cdot)$ to $\hat{D} \times  \hat{D}.$
Since $K_1=K$ on $D_1$ and $K_2=K$ on $D_2$, it follows from uniqueness of sesqui-holomorphic functions and connectedness of $D_1, D_2$ that
%\[
%K_1(z,\bar z)=\widetilde K(z,\bar z)\quad \text{for } z\in D_1,
%\qquad
%K_2(z,\bar z)=\widetilde K(z,\bar z)\quad \text{for } z\in D_2.
%\]
\[
K_1(z,w)=\widetilde K(z,w)\quad \text{on } D_1\times D_1,
\qquad
K_2(z,w)=\widetilde K(z,w)\quad \text{on } D_2\times D_2.
\]
In particular,
\begin{equation}\label{eqnk12zw}
K_1(z,w)=K_2(z,w)
\quad \text{for all } z,w\in D_1\cap D_2.
\end{equation}
Let $\mathcal H_j:=A^2(D_j)$ for $j=1,2$, and define
\[
\mathcal E_j:=\mathrm{Span}\{K_j(\cdot, w): w\in D_1\cap D_2\}\subseteq \mathcal H_j.
\]
We define a linear operator $\mathcal J:\mathcal E_1 \to \mathcal E_2$ as follows.
For
\begin{equation}
	\phi=\sum_{j=1}^s c_j K_1(\cdot, w_j) \in \mathcal E_1,~~~\text{where } s \in \Z_+, w_j \in D_1\cap D_2, \text{ and } c_j \in \C,
\end{equation}
we set
$\mathcal J \phi:=\sum_{j=1}^s c_j K_2(\cdot, w_j) \in \mathcal E_2.$
It follows from \eqref{eqnk12zw} that $\mathcal J$ is well defined.
By a similar argument as in the proof of Proposition \ref{prop:rmk15}, we conclude that
%define
%\[
%\mathcal J \bigl(K_1(\cdot, w)\bigr):=K_2(\cdot, w),\qquad w\in D_1\cap D_2.
%\]
%and then extend by linearity.  Moreover, we claim 
$\mathcal J$ is surjective and isometric from $\mathcal E_1$ to $\mathcal E_2$;  %To see that, by the reproducing property of the Bergman kernels, we have for $z , w \in D_1\cap D_2$,
%We first show that $U_0$ is well defined. Suppose
%\[
%\sum_{j=1}^N c_j K_1(\cdot,\overline{w_j})=0
%\]
%for some pairwise distinct points $w_1,\cdots,w_N\in D_1\cap D_2$. 
%Since $D_1$ is bounded, polynomials belong to $A^2(D_1)$. 
%Choose polynomials $p_m$ such that $p_m(w_j)=\delta_{mj}$. Then
%\[
%0=\Big\langle p_m,\sum_{j=1}^N c_j K_1(\cdot,\overline{w_j})\Big\rangle
%=\sum_{j=1}^N c_j p_m(w_j)=c_m,
%\]
%so all $c_j=0$. Hence $U_0$ is well defined.
%Moreover, $U_0$ is isometric. For $w,\zeta\in D_1\cap D_2$,
%\[
%\langle K_1(\cdot, w),K_1(\cdot, \zeta)\rangle_{\mathcal H_1}
%=K_1(\zeta, w)
%=K_2(\zeta, w)
%=\langle K_2(\cdot,w),K_2(\cdot,\zeta)\rangle_{\mathcal H_2}.
%\]
and $\mathcal E_j$ is dense in $\mathcal H_j, j=1,2$. %If $f\in \mathcal H_1$ is orthogonal to $\mathcal E_1$, then 
%$$f(w)=\langle f(\cdot),K_1(\cdot, w)\rangle_{\mathcal H_1}=0~\text{for all}~w\in D_1\cap D_2.$$ 
%Since $D_1\cap D_2$ is dense in $D_1$ and $f$ is holomorphic, $f\equiv 0$. Thus $\mathcal E_1$ is dense, and similarly for $\mathcal E_2$.
As a consequence, $\mathcal J$ extends uniquely to a unitary operator from $\mathcal H_1$ onto $\mathcal H_2$, which we continue to denote by $\mathcal J$.

Let $f\in \mathcal H_1$ and $w\in D_1\cap D_2$. Then
\[
(\mathcal{J}f)(w)=\langle \mathcal{J} f, K_2(\cdot,w)\rangle
=\langle f,K_1(\cdot, w)\rangle
=f(w).
\]
Here the second equality follows from the fact that $\mathcal{J}$ is unitary and that $K_2(\cdot,w)=\mathcal{J}K_1(\cdot,w)$. This shows $\mathcal Jf=f$ on $D_1\cap D_2$.
Consequently,
\[
F(z):=
\begin{cases}
	f(z), & z\in D_1,\\
	\mathcal Jf(z), & z\in D_2.
\end{cases}
\]
is well defined and holomorphic on $D=D_1\cup D_2$. Moreover,
\[
\int_D |F|^2
=
\int_{D_1} |f|^2
+
\int_{D_2\setminus D_1} |\mathcal Jf|^2.
\]
Since $D_2\setminus D_1$ has measure zero, the second term vanishes, so $\|F\|_{A^2(D)}=\|f\|_{A^2(D_1)}$.
Hence $D\setminus D_1$ is Bergman negligible, so $D_1\leqB D$. By symmetry $D_2\leqB D$, and thus $D_1\sim_B D_2$. \qed

%%%%%%%%%%%%%%%%%%%%%%%%%%%%%%%%%%%%%%%%%%%%%%%%%%%%%%%%%%%%%%%%%%%
%%%%%%%%%%%%%%%%%%%%%%%%%%%%%%%%%%%%%%%%%%%%%%%%%%%%%%%%%%%%%%%%%%%
%%%%%%%%%%%%%%%%%%%%%%%%%%%%%%%%%%%%%%%%%%%%%%%%%%%%%%%%%%%%%%%%%%%
%%%%%%%%%%%%%%%%%%%%%%%%%%%%%%%%%%%%%%%%%%%%%%%%%%%%%%%%%%%%%%%%%%%
%%%%%%%%%%%%%%%%%%%%%%%%%%%%%%%%%%%%%%%%%%%%%%%%%%%%%%%%%%%%%%%%%%%
%%%%%%%%%%%%%%%%%%%%%%%%%%%%%%%%%%%%%%%%%%%%%%%%%%%%%%%%%%%%%%%%%%%
%%%%%%%%%%%%%%%%%%%%%%%%%%%%%%%%%%%%%%%%%%%%%%%%%%%%%%%%%%%%%%%%%%%
%%%%%%%%%%%%%%%%%%%%%%%%%%%%%%%%%%%%%%%%%%%%%%%%%%%%%%%%%%%%%%%%%%%

\end{document}